\documentclass{amsart} 

\usepackage[active]{srcltx}

\newtheorem{definition}{Definition}[section]

\newtheorem{remark}[definition]{Remark}
\newtheorem{example}[definition]{Example}
\newenvironment{ack}{\noindent{\bf Acknowledgements}.}{}

\newtheorem{lemma}[definition]{Lemma}
\newtheorem{proposition}[definition]{Proposition}
\newtheorem{theorem}[definition]{Theorem}

\def\P{{\mathbb{P}}}

\def\K{{\mathbb{K}}}
\def\kK{{\mathcal{K}}}
\def\N{{\mathbb{N}}}

\def\cE{{\mathcal{E}}}
\def\cG{{\mathcal{G}}}
\def\cF{{\mathcal{F}}}
\def\cM{{\mathcal{M}}}
\def\cP{{\mathcal{P}}}

\def\T{{\underline{T}}}
\def\t{{\underline{t}}}

\def\X{{\underline{X}}}

\def\u{{\underline{u}}}
\def\Z{{\mathbb{Z}}}
\def\bv{{\underline{v}}}
\def\bM{{\mathbb{M}}}

\def\C{{\mathrm{C}}}

\def\S{{\underline{S}}}

\def\bdeg{{\mbox{bideg}}}

\def\cD{{\mathcal{D}}}
\def\cG{{\mathcal{G}}}

\begin{document}
\title[Generators of the Rees Algebra of a parametrization with $\mu=2$]{Minimal generators of the defining ideal of the Rees Algebra associated to a rational plane parametrization with $\mu=2$}

\author{Teresa Cortadellas Ben\'itez}
\address{Universitat de Barcelona, Facultat de Formaci\'o del Professorat.
Passeig de la Vall d'Hebron 171,
08035 Barcelona, Spain}
\email{terecortadellas@ub.edu}

\author{Carlos D'Andrea}
\address{Universitat de Barcelona, Facultat de Matem\`atiques.
Gran Via 585, 08007 Barcelona, Spain} \email{cdandrea@ub.edu}
\urladdr{http://atlas.mat.ub.es/personals/dandrea}
\thanks{Both authors are supported by the Research Project MTM2010--20279 from the
Ministerio de Ciencia e Innovaci\'on, Spain}

\subjclass[2010]{Primary 13A30; Secondary 14H50}

\begin{abstract}
We exhibit a set of minimal generators of the defining ideal of the Rees Algebra associated to the ideal of three bivariate homogeneous polynomials parametrizing a proper rational curve in projective plane, having  a minimal syzygy of degree $2.$
\end{abstract}
\maketitle

\section{Introduction}\label{s0}
Let $\K$ be an algebraically closed field, and $u_0(T_0,T_1),\,u_1(T_0, T_1),\, u_2(T_0,T_1)\in\K[T_0,T_1]$ homogeneous polynomials of the same degree $d\geq1$
without common factors. Denote with $\T$ the sequence $T_0, T_1,$ set $R:=\K[\T]$ and $I:=\langle
u_0(\T),u_1(\T),u_2(\T) \rangle$ for the ideal generated by these polynomials in $R$. The {\em Rees
Algebra} associated to $I$ is defined as $\mbox{Rees}(I):=\bigoplus_{n\geq0}I^nZ^n,$  where
$Z$ is a new variable. Let $X_0,\,X_1,\,X_2$ be another three variables and set $\X=X_0,\,X_1,\,X_2.$ There is a graded epimorphism of $K[\T]$-algebras
defined by
\begin{equation}\label{ris}
\begin{array}{cccc}
   {\Phi}:&\K[\T][\X]&\to&\mbox{Rees}({I})\\
&X_i&\mapsto&u_i(\T)\,Z.
  \end{array}
\end{equation}
Set $\kK:=\ker ({\Phi})$. Note that a description of $\kK$ allows also a full characterization of
$\mbox{Rees}(I)$ via \eqref{ris}. This is why we call it \emph{the defining ideal of the Rees Algebra} associated to $I$.
\par
The search for explicit generators of $\kK$ is an active area of research in both the Commutative Algebra and also the Computer Aided Geometric Design community. Indeed, the defining polynomials of $I$ induce a rational map
\begin{equation}\label{param}
\begin{array}{cccc}
\phi:&\P^1&\to&\P^2\\
&(t_{0}:t_{1})&\mapsto&(u_0(t_{0},t_{1}):u_1(t_{0},t_{1}):u_2(t_{0},t_{1})).
  \end{array} \end{equation}
whose image is an irreducible algebraic plane curve $\C,$ defined by the zeros of a homogeneous irreducible element of $\K[X_0,X_1,X_2].$ This polynomial can be
computed easily by applying elimination techniques on the input parametrization, but it is easy to see that the elimination can also be applied on any suitable pair of minimal elements in $\kK$,  leading to better algorithms for computing invariants associated to $\phi$. This is why finding such elements are of importance in the CAGD community, see for instance \cite{SC95,SGD97,CSC98,ZCG99,CGZ00,cox08}.
\par A lot of progress has been made in the last years: a whole description of $\kK$ has been given in the case when  $\C$ has a point of maximal multiplicity in \cite{CHW08,bus09,CD10}, an extension of this situation to ``de Jonqui\` eres parametrizations'' is the subject of \cite{HS12}. In \cite{bus09},  a detailed description of generators  of $\kK$ via inertia forms associated to the syzygies of $I$ is done; the case when $\phi$ has an inverse of degree $2$ is the subject of \cite{CD13}; extensions to surfaces and/or non planar curves have also been considered in \cite{CCL05, CD10, HW10, KPU09};  connections between singularities and minimal elements in $\kK$ are studied in \cite{CKPU11, KPU13}.
\par\smallskip
In this paper, we give  a complete description of a minimal set of generators of the defining ideal of the Rees Algebra associated to $I$ in the case when there is a minimal syzygy of $I$ of degree $2$ (in the language of $\mu$-bases, this means just $\mu=2$). Our main results are given in Sections \ref{sing} and \ref{nosing}, where we make explicit these generators in two different cases: when there is a singular point of multiplicity $d-2$ (Theorems \ref{mt1o} for $d$ odd and \ref{mt1e} for $d$ even), and when all the singularities are double points (Theorem \ref{mt2}). The latter situation is just a refinement of \cite[Proposition $3.2$]{bus09}, where an explicit set of generators of $\kK$ is actually given. Our contribution in this case is to show that Bus\'e's family is essentially minimal: there is only one member in this family which can be removed from the list and yet generate the whole $\kK$.

\par  There is some general evidence that the more complicated the singularity, the simpler the description of $\mbox{Rees}(I)$ should be, see for instance \cite{CKPU11}. The situation for $\mu=2$ is not an exception, indeed from Theorems \ref{mt1o},\,\ref{mt1e} and \ref{mt2} we easily derive that the number of minimal generators of $\kK$  is of the order $\mathcal{O}\left(\frac{d}2\right)$ in the case of a very singular point, and
$\mathcal{O}\left(\frac{d^2}2\right)$ otherwise. Note also that the generators we present in the very singular point case are not specializations of the larger family produced in \cite{bus09} (which it was shown in that paper that they are always elements of $\kK$), but they actually appear at lower bidegrees. Moreover, we show in Section \ref{adjoints} that in the very singular case not all the elements in $\kK_{1,*}$ are pencils of adjoints, as shown also by Bus\'e in the other case. We should mention that a few days after we posted a preliminary version of these results (\cite{CD13b}) in the arxiv, the article \cite{KPU13} was uploaded in the same database. In that work, the authors get the same description we achieved in Section \ref{sing}  with  a  refined kit of tools from local cohomology and linkage.

\smallskip
The paper is organized as follows. In Section \ref{s1} we review some well-known facts about elements in $\kK$, and focus in the case where the curve $\C$ has a very singular point.  We detect in Theorem \ref{mtv} a special family which is part  of a  minimal set of generators of $\kK$. The rest of the paper focuses on the case $\mu=2$. In Section \ref{sing} we show that if the curve has a very singular point, we only have to add  one (if $d$ is odd) or two (if $d$ is even) elements to this special family to get a whole set of minimal generators of $\kK.$ This is the content of Theorems \ref{mt1o} ($d$ odd) and \ref{mt1e} ($d$ even). 

\par We then  introduce pencils of adjoints in Section \ref{adjoints}, and show in Theorem \ref{ttt} that $\kK_{1,*}$ strictly contains the subspace of pencils of adjoints in the case of a very singular point. The other case has already been studied in \cite{bus09}. 

\par Section \ref{nosing} deals with the case when all the singularities are mild (i.e. no multiplicity larger than two). In this case, we show in Theorem \ref{mt2} that Bus\'e's family of generators of $\kK$ given in \cite[Proposition $3.2$]{bus09} is essentially minimal in the sense that there is only one of them that can be removed from the list. The paper concludes with a brief discussion of how these methods may not work for larger values of $\mu$ in Section \ref{final}.

\smallskip
\begin{ack}
We are grateful to Eduardo Casas-Alvero for several discussions on adjoint curves. All our computations and experiments were done with the aid of the softwares {\tt Macaulay 2} \cite{mac} and {\tt Mathematica} \cite{math}.
\end{ack}
\bigskip
\section{Preliminaries on Rees Algebras and singularities}\label{s1}
Set  $\u(\T):=(u_0(\T),\,u_1(\T),\,u_2(\T))$ for short. By its definition, $\kK\subset\K[\T,\X]$ is a bihomogeneous ideal,
 which can be characterized as follows:
\begin{equation}\label{oso}
P(\T,\X)\in \kK_{i,j}\iff \bdeg(P(\T,\X))=(i,j)\ \mbox{and} \ P(\T,\u(\T))=0.
\end{equation}
There is a natural identification of
$\kK_{\ast,1}$ with $\mbox{Syz}(I),$ the first module of syzygies of $I$.
A straightforward application of  Hilbert Syzygy Theorem shows  that $\mbox{Syz}(I)$ is a free $R$-module of rank $2$ generated by two
elements, one of $\T$-degree $\mu$ for an integer $\mu$ such
that $0\leq\mu\leq\frac{d}{2}$, and the other of $\T$-degree $d-\mu$.
In the Computer Aided Geometric Design community, such a basis is called a $\mu$-basis of $I$ (see for instance \cite{CSC98,CGZ00,CCL05}). Indeed, by the
Hilbert-Burch Theorem, $I$ is generated by the maximal
minors of a $3\times 2$ matrix $\varphi$ and  the
 homogeneous resolution of $I$ is
\begin{equation}\label{sec}
0\longrightarrow R(-d-\mu)\oplus R(-d-(d-\mu))
\stackrel{\varphi}{\longrightarrow} R(-d)^3
\stackrel{(u_0,u_1,u_2)}{\longrightarrow} I\rightarrow  0.\end{equation}
This matrix is called the Hilbert-Burch matrix of $I$ and its
columns describe the $\mu$-basis. In the sequel, we
will denote with $P_{\mu ,1}(\T,\X),\,Q_{d-\mu ,1}(\T,\X)\in
\kK_{*,1}$ a (chosen) set of two elements in $\mbox{Syz}(I)$ which are a basis of this module over $R.$
\par
All along this paper we will work under the assumption that the map $\phi$ defined in \eqref{param} is ``proper'', i.e. birational. If this is not the case, then by L\"uroth's Theorem one can prove that $\phi$ is the composition of a proper map $\overline{\phi}:\P^1\to\P^2$ with a polynomial automorphism $\underline{p}:\P^1\to\P^1,$ and our results can be easily translated to this case.

\par The following statements have been proven in  \cite{CD13}. We will use them in the sequel. 
\begin{proposition} (\cite[Section $1$ and Lemma $3.10$]{CD13})\label{ssing}
Let $\phi$ be as in \eqref{param} be a proper parametrization of a rational plane curve $\C$, and let $T_0\mathcal{B}_\ell(\X)-T_1\mathcal{A}_\ell(\X)\in\kK_{1,\ell}$ a non zero element. Then, the map
$$\begin{array}{cccc}
\psi:&\C&\dashrightarrow&\P^1\\
&(x_0:x_1:x_2)&\mapsto&\left(\mathcal{A}_\ell(x_0,x_1,x_2):\mathcal{B}_\ell(x_0,x_1,x_2)\right)
\end{array}
$$
 is an inverse of $\phi$. Moreover, the singularities of $\C$ are contained in the set of common zeroes of $\{\mathcal{A}_\ell(\X),\,\mathcal{B}_\ell(\X)\}$ in $\P^2.$ Reciprocally, any inverse of $\phi$ induces a non zero element in $\kK_{1,\ell}$ via the correspondence shown above, with $\ell$ being the degree of the polynomials defining $\phi^{-1}.$
\end{proposition}

Denote with $\cE_d(\X)$ the irreducible polynomial of degree $d$ defining $\C,$ it is a primitive element generating $\kK\cap\K[\X].$ Note that it is well-defined up to a nonzero constante in $\K.$

\begin{proposition}(\cite[Proposition $4.1$]{CD13})\label{multiple}
Suppose $T_0\cF^1_{k_0}(\X)-T_1\cF^0_{k_0}(\X)\in\kK_{1,k_0}$ for some $k_0\in\N.$ Then, $G_{i,j}(\T,\X)\in\kK_{i,j}$ if and only if
$G_{i,j}(\cF^0_{k_0}(\X),\cF^1_{k_0}(\X),\X)$ is a multiple of $\cE_d(\X)$.
\end{proposition}

\begin{theorem}(\cite[Theorem $4.6$]{CD13})\label{mu}
Let $u_0(\T),\,u_1(\T),\,u_2(\T)\in\K[\T]$ be homogeneous polynomials of
degree $d$ having no common factors. A minimal set of
generators of $\kK$ can be found with all its elements having
$\T$-degree  strictly less than $d-\mu$ except for the
generators of $\kK_{\ast,1}$ with $\T$-degree $d-\mu$.
\end{theorem}

\subsection{Curves with very singular points}\label{s2}

\begin{definition}
Let $\mu$ be the degree of the first non-trivial syzygy of $I$.
A point ${\bf p}\in\C$ in is said to be {\em very singular} if $mult_P(\C)>\mu.$
\end{definition}

The following result is an extension of \cite[Theorem $1$]{CWL08}.  Recall that we have fixed a basis of the $\K[\X]$-module $\mbox{Syz}(I)$ which we denote with
$\{P_{\mu,1}(\T,\X),\,Q_{d-\mu,1}(\T,\X)\}.$
\begin{proposition}\label{kkk}
A rational plane curve $\C$ can have at most one very singular point. If this is the case, then after a linear change of the $\X$ variables, one can write
\begin{equation}\label{form}
P_{\mu ,1}(\T,\X)=p^1_\mu(\T)X_0-p^0_\mu(\T)X_1.
\end{equation}
Reciprocally, if $2\mu<d$ and after a linear change of $\X$-coordinates $P_{\mu,1}(\T,\X)$ has a form like \eqref{form}, then $\C$ has ${\bf p}=(0:0:1)$ as its only very singular point.
\end{proposition}
\begin{proof}
The first part of the claim follows directly from \cite[Theorem 1]{CWL08}. For the converse, note that if $P_{\mu,1}(\T,\X)$ is like \eqref{form}, then by computing $\u(\T)$ from the Hilbert Burch matrix appearing in \eqref{sec}, we will have
\begin{equation}\label{riso}
\begin{array}{ccc}
u_0(\T)&=&p^0_\mu(\T)\,q(\T),\\
u_1(\T)&=&p^1_\mu(\T)\,q(\T),
\end{array}
\end{equation}
for a homogeneous polynomial $q(\T)\in\K[\T]$ of degree $d-\mu>\mu.$ Hence, the preimage of the point $P=(0:0:1)$ has $d-\mu$ values counted with multiplicities (the zeroes of $q(\T)$), and so we get
$mult_P(\C)>\mu.$
\end{proof}
\smallskip
\begin{remark}
Note that if $\C$ has $(0:0:1)$ as a very singular point, then 
\begin{equation}\label{cucco}
\cE_d(\X)=\cE^0_{d}(X_0,X_1)+\cE^1_{d-1}(X_0,X_1)\,X_2+\ldots +\cE^\mu_{d-\mu}(X_0,X_1)\,X_2^\mu,
\end{equation}
with $\cE^i_{d-i}(X_0,X_1)\in\K[X_0,X_1],$ homogeneous of degree $i,$ and $\cE^\mu_{d-\mu}(X_0,X_1)\neq0.$
\end{remark}
The syzygy $P_{\mu,1}(\T,\X)$ in \eqref{form} is called an {\em axial moving line} around $P$ in
\cite{CWL08}. The following result is well-known, and will be used in the sequel
\begin{proposition}\label{lem2}
Let $a_{s_0}(\T),\,b_{s_0}(\T)\in\K[\T]$ homogeneous of the same degree $s_0$ without common factors. Then, $\left( a_{s_0}(\T),\,b_{s_0}(\T)\right)_s=\K[\T]_s$ for $s\geq2s_0-1.$
\end{proposition}
\begin{proof}
By hypothesis, the classical Sylvester resultant of $a_{s_0}(\T)$ and $b_{s_0}(\T)$ (for its definition, see for instance \cite{CLO07}) is not zero, and moreover from the Sylvester matrix which computes this resultant, we can get a B\'ezout identity of the form
$$\tilde{a}^{j}_{s_0-1}(\T)a_{s_0}(\T)+\tilde{b}^j_{s_0-1}(\T)b_{s_0}(\T)=Res_{\T}\big(a_{s_0}(\T), b_{s_0}(\T)\big)\,T_0^{j}T_1^{2s_0-1-j}
$$
for $j=0, 1, \ldots, 2s_0-1.$ This shows that $\left( a_{s_0}(\T),\,b_{s_0}(\T)\right)_{2s_0-1}=\K[\T]_{2s_0-1},$ and the rest of the claim follows straightforwardly from here.
\end{proof}

Several of the proofs in this text will be done by induction on degrees. In order to be able to pass from one degree to another, we will apply a pair of operators, one which decreases the degree in $\T$ and another which does it with $\X$. Recall from \eqref{form} that we have $P_{\mu ,1}(\T,\X)=p^1_\mu(\T)X_0-p^0_\mu(\T)X_1.$
\begin{definition}\label{DT}
If $G_{i,j}(\T,\X)\in\K[\T,\X]_{i,j},$ with $i\geq2\mu-1,$ then write
$$G_{i,j}(\T,\X)=p^0_\mu(\T)G^0_{i-\mu,j}(\T,\X)+p^1_\mu(\T)G^1_{i-\mu,j}(\T,\X),
$$
and set
$$\cD_T\big(G_{i,j}(\T,\X)\big):=X_0G^0_{i-\mu,j}(\T,\X)+X_1G^1_{i-\mu,j}(\T,\X)\in\K[\T,\X]_{i-\mu,j+1}.
$$
If $G_{i,j}(\T,\X)\in\K[\T,\X]_{i,j}\cap\langle X_0,\,X_1\rangle,$ then write
$$G_{i,j}(\T,\X)=X_0G^0_{i,j-1}(\T,\X)+X_1G^1_{i,j-1}(\T,\X),
$$
and set
$$\cD_X\big(G_{i,j}(\T,\X)\big):=p^0_\mu(\T)G^0_{i,j-1}(\T,\X)+p^\mu_1(\T)G^1_{i,j-1}(\T,\X)\in\K[\T,\X]_{i+\mu,j-1}.
$$
\end{definition}
Note that both operators are in principle not well defined as the decomposition of $G_{i,j}(\T,\X)$ given above is not necessarily unique. In the next proposition we show that it is actually well defined modulo $P_{\mu,1}(\T,\X).$
\begin{proposition}\label{tec1}
both $\cD_T(G_{i,j}(\T,\X))$ and $\cD_X(G_{i,j}(\T,\X))$ are well defined modulo $P_{\mu,1}(\T,\X).$ Moreover, the image of $\cD_\T$ lies in the ideal $\langle X_0,\,X_1\rangle,$ and
\begin{equation}\label{ident}
\cD_X\left(\cD_T\big(G_{i,j}(\T,\X)\big)\right)=G_{i,j}(\T,\X)\ \mod P_{\mu,1}(\T,\X).\end{equation}
In addition, if $G_{i,j}(\T,\X)\in\kK,$ then both $D_T(G_{i,j}(\T,\X))$ and $D_X(G_{i,j}(\T,\X)),$ when defined, are also elements of $\kK.$
\end{proposition}
\begin{proof}
Consider first $\cD_\T,$ so it is enough to show that if $i\geq2\mu-1$ and
\begin{equation}\label{hhuno}
p^0_\mu(\T)G^0_{i-\mu,j}(\T,\X)+p^1_\mu(\T)G^1_{i-\mu,j}(\T,\X)=0,
\end{equation} then
$X_0G^0_{i-\mu,j}(\T,\X)+X_1G^1_{i-\mu,j}(\T,\X)$ is a multiple of $P_{\mu,1}(\T,\X).$ But from \eqref{hhuno}, we get
$$\begin{array}{ccl}
G^0_{i-\mu,j}(\T,\X)&=&p^1_{\mu}(\T)H_{i-2\mu,j}(\T,\X)\\
G^1_{i-\mu,j}(\T,\X)&=&-p^0_{\mu}(\T)H_{i-2\mu,j}(\T,\X),
\end{array}
$$
with $H_{i-2\mu,j}(\T,\X)\in\K[\T,\X],$ and hence $$X_0G^0_{i-\mu,j}(\T,\X)+X_1G^1_{i-\mu,j}(\T,\X)=P_{\mu,1}(\T,\X)\,H(\T,\X).$$
The proof of the claim for $\cD_X$ and for the composition $\cD_X\circ\cD_T$ follows analogously.
\par
To conclude, suppose $G_{i,j}(\T,\X)\in\kK$ with $i\geq2\mu-1.$ Due to \eqref{oso}, this is equivalent to having
$$G_{i,j}(\T,\u(\T))=p^0_\mu(\T)G^0_{i-\mu,j}(\T,\u(\T))+p^1_\mu(\T)G^1_{i-\mu,j}(\T,\u(\T))=0.
$$
From here, by using \eqref{riso}, we get immediately that
$$
\left.\cD_T\big(G_{i,j}(\X,\T)\big)\right|_{\X\mapsto\u(\T)}=q(\T)\left(p^0_\mu(\T)G^0_{i-\mu,j}(\T,\u(\T))+p^1_\mu(\T)G^1_{i-\mu,j}(\T,\u(\T))\right)=0,
$$
which shows that $\cD_T\big(G_{i,j}(\X,\T)\big)\in\kK$, again by \eqref{oso}. The proof for $\cD_X\big(G_{i,j}(\T,\X)\big)$ follows analogously.
\end{proof}

\bigskip
\subsection{Elements of low degree in $\kK$}
We will assume here that $\mu<d-\mu,$ and set $d=k\mu+r,$ with $k\in\N$ and $-1\leq r<\mu-1,$ i.e. $k$ and $r$ are the quotient and remainder respectively of the division between $d$ and $\mu$, except in the case when $d+1$ is a multiple of $\mu$.

With this information, we will produce minimal generators of $\mbox{Rees}(I),$ in the case where the curve $\C$ defined by the generators $\u(\T)$ of $I$ has a very singular point, which we will assume w.l.o.g. that it is $P=(0:0:1).$

We start by setting
 $$\boxed{F_{\mu,1}(\T,\X):=P_{\mu,1}(\T,\X),\,F_{(k-1)\mu+r,1}(\T,\X):=Q_{d-\mu,1}(\T,\X),}$$
a basis of the syzygy module of $I$. Note that we have $(k-1)\mu+r=d-\mu.$

Now for $j=2,\ldots, k-1$ we will define recursively $F_{(k-j)\mu+r,j}(\T,\X)\in\kK$ as follows:
\begin{equation}\label{orozco}
\boxed{F_{(k-j)\mu+r,j}(\T,\X)=\cD_T\big(F_{(k-j+1)\mu+r,j-1}(\T,\X)\big).}
\end{equation}
Note that we can apply the operator $\cD_T$ to these polynomials as their $\T$-degree is $(k-j+1)\mu+r\geq2\mu-1.$ Also, we have to make a choice in order to define each of these polynomials, but we know that they are all equivalent modulo $F_{\mu,1}(\T,\X)$ thanks to Proposition \ref{tec1}.

\begin{theorem}\label{mtv} $^{}$
\begin{enumerate}
\item For each $j=1,\ldots, k-1,\, F_{(k-j)\mu+r,j}(\T, \X)\in\kK$ and it is not a multiple of $F_{\mu,1}(\T,\X).$ In particular, it is not identically zero.
\item Up to a nonzero constant in $\K$, we have
$$\mbox{Res}_\T\left(F_{\mu,1}(\T, \X),\,F_{(k-j)\mu+r,j}(\T, \X)\right)=\cE_d(\X), \ \ \ j=1, 2, \ldots, k-1.
$$
\item  If $G_{i,j}(\T,\X)\in\kK_{i,j}$ with $i+\mu\,j<d,$ then $G_{i,j}(\T, \X)$ is a multiple of $F_{\mu,1}(\T,\X).$
\item The set of $k+1$ elements
\begin{equation}\label{family}
\{\cE_d(\X),\,F_{\mu,1}(\T, \X),\,F_{\mu+r,k-1}(\T, \X),\,F_{2\mu+r, k-2}(\T, \X),\,\ldots, F_{d-\mu,1}(\T, \X)\}
\end{equation}
is part of a minimal system of generators of $\kK$.
\end{enumerate}
\end{theorem}

\begin{proof} $^ {}$
\begin{enumerate}
\item By induction on $j,$ the case $j=1$ being obvious. Suppose then $j>1.$ Due to Propositon \ref{tec1}, we know that
$$F_{(k-j)\mu+r,j}(\T,\X)=\cD_T\big(F_{(k-j+1)\mu+r,j-1}(\T,\X)\big)\in\kK.$$ Note also that by construction, we have straightforwardly
$$X_1F_{(k-(j-1))\mu+r,j-1}(\T,\X)-p^1_{\mu}(\T)F_{(k-j)\mu+r,j}(\T, \X)\in\langle F_{\mu,1}(\T, \X)\rangle.
$$
If $F_{(k-j)\mu+r,j}(\T, \X)$ is a multiple of $F_{\mu,1}(\T, \X)$, then as the latter is irreducible, we would then conclude that $F_{(k-(j-1))\mu+r,j-1}(\T,\X)$ is also a multiple of this polynomial, which again contradicts the inductive hypothesis.
\item Clearly $\mbox{Res}_\T\left(F_{\mu,1}(\T, \X),\,F_{(k-j)\mu+r,j}(\T, \X)\right)\in\K[\X].$ Moreover, an explicit computation reveals that
the $\X$-degree of this resultant is equal to $k\mu+r=d,$ which is the degree of $\cE_d(\X).$ So, it must be equal to $\lambda\,\cE_d(\X)$ with $\lambda\in\K$. If $\lambda=0,$ this would imply that both $\{F_{\mu,1}(\T, \X),\,F_{(k-j)\mu+r,j}(\T, \X)\}$ have a non trivial common factor in $K[\T,\X]$. But  $F_{\mu,1}(\T,\X)$ is irreducible, and we just saw in (1) that
$F_{(k-j)\mu+r,j}(\T, \X)$ is not a multiple of it, which then shows that the resultant cannot vanish identically, so $\lambda\neq0.$
\item We have
$$\mbox{Res}_\T\left(F_{\mu,1}(\T, \X),\,G_{i,j}(\T, \X)\right)=\cE_d(\X)\,\alpha_{\mu\,j+i-d}(\X),$$
so in order to have this resultant different from zero, we must have
$0\leq \mu\,j+i-d,$ contrary to our hypothesis. Hence, the resultant above vanishes identically, and due to the irreducibility of $F_{\mu,1}(\T, \X),$ we have that $G_{i,j}(\T, \X)$ must be a multiple of it.
\item  Clearly $F_{\mu,1}(\T,\X)$ is minimal in this set, so it cannot be a combination of the others. Also, the family
$$\{F_{\mu+r,k-1}(\T, \X),\,F_{2\mu+r, k-2}(\T, \X),\,\ldots, F_{d-\mu,1}(\T, \X)\}$$ is pseudo-homogeneous with weighted degree $\deg_\T+\mu\,\deg_\X=d$ (i.e. all the exponents lie on a line). This shows that none of the elements in this family can be a combination of the others, and as we have seen in (1), none of them is a multiple of $F_{\mu,1}(\T,\X)$, so
this is a minimal set of generators of the ideal they generate. To see that they can be extended to a whole set of generators of $\kK,$ consider the maximal ideal
$\mathfrak{M}=\langle \T,\,\X\rangle$ of $R$. The pseudo-homogeneity  combined with (1) and (3) implies straightforwardly that the family \eqref{family} is $\K$-linearly independent in the quotient  $\kK/\mathfrak{M}\kK.$  By the homogeneous version of Nakayama's lemma (see for instance \cite[Exercise $1.5.24$]{BH93}), we can extend this family to a minimal set of generators of $\kK$. This completes the proof.
\end{enumerate}
\end{proof}
\smallskip
\begin{remark}
If $\mu=1$, then one can take $k=d$ or $k=d+1$. If we choose $k=d,$ then it is easy to see that the family \eqref{family} actually specializes in the minimal set of generators of $\kK$ described in \cite[Theorem $2.10$]{CD13}. So, this construction may be regarded somehow as a generalization of the tools used in \cite{CD13} for the case $\mu=1.$
\end{remark}

\bigskip
\section{The case $\mu=2$ with $\C$ having a very singular point}\label{sing}

\subsection{$d$ odd}
In this case, we will show that the family given in Teorem \ref{mtv} (4) is ``almost'' a minimal set of generators of $\kK$. We only need to add one more element  of bidegree $(1,\frac{d+1}2)$ to the list in order to generate the whole $\kK.$
Suppose then in this paragraph that $\mu=2$, and $d=2k-1,$ with $k\in\N,\,k>2$ (otherwise $\mu=1$). Note that in this case, there is a form of $\T$-degree one in \eqref{family}.
We will define an extra element in $\kK$ by computing the so called {\em Sylvester form} among $F_{1,k-1}(\T,\X)$ and $F_{2,1}(\T,\X)$. This process is standard in producing nontrivial elements in $\kK$, see for instance \cite{BJ03, bus09,CD10, CD13}:
\begin{itemize}
\item Write $F_{2,1}(\T,\X)=T_0G_{1,1}(\T,\X)+T_1H_{1,1}(\T,\X),$ with $G_{1,1}(\T,\X),\,H_{1,1}(\T,\X)\in\K[\T,\X].$ Note that this decomposition is not unique.
\item Write $F_{1,k-1}(\T,\X)=T_0\cF^1_{k-1}(\X)-T_1\cF^0_{k-1}(\X),$ with $\cF^i_{k-1}(\X)\in\K[\X],$ homogeneous of degree $d-1.$
\item Set \begin{equation}\label{ddef}
\boxed{F_{1,k}(\T,\X):=\cF^0_{k-1}(\X)G_{1,1}(\T,\X)+\cF^1_{k-1}(\X)H_{1,1}(\T,\X)}.\end{equation}
\end{itemize}
\smallskip
The following claims will be useful in the sequel.
\begin{lemma}\label{loma}
$F_{1,k}(\T,\X)\in\kK_{1,k}\setminus\langle F_{1,k-1}(\T,\X)\rangle,$ in particular it is not identically zero.
\end{lemma}
\begin{proof}
By construction, we have $$\begin{array}{ccl}
F_{1,k}(\cF^0_{k-1}(\X),\cF^1_{k-1}(\X),\X)&=&F_{2,1}(\cF^0_{k-1}(\X), \cF^1_{k-1}(\X),\X)\\
&=&\pm\mbox{Res}_\T(F_{2,1}(\T,\X), F_{1,k-1}(\T,\X))=\pm\cE_d(\X),
\end{array}$$
the last equality due to  Theorem \ref{mtv} (2). By Proposition \ref{multiple}, we then conclude that $F_{1,k}(\T,\X)\in\kK_{1,k},$ and it is clearly nonzero. Moreover, as both
$F_{1,k}(\T,\X)$ and $F_{1,k-1}(\T,\X)$ have degree $1$ in $\T,$ the fact that $F_{1,k}\big(\cF^0_{k-1}(\X), \cF^1_{k-1}(\X),X\big)\neq0$ implies that they are $\K$-linearly independent, and from here the rest of the claim follows straightforwardly.
\end{proof}
\smallskip
\begin{lemma}\label{k-1}
$F_{1,k}(\T,\X)\in\langle X_0,X_1\rangle,$ and modulo $F_{2,1}(\T,\X),$ we have
\begin{equation}\label{kk-1}
D_X(F_{1,k}(\T,\X))\in\langle F_{1,k-1}(\T,\X)\rangle.
\end{equation}
\end{lemma}
\begin{proof}
Write as before $F_{2,1}(\T,\X)=T_0G_{1,1}(\T,\X)+T_1H_{1,1}(\T,\X),$ and note that as $F_{2,1}(\T,\X)\in\K[\T,X_0,X_1],$ we then have
$G_{1,1}(\T,\X)=G_{1,1}(\T,X_0,X_1)$ and also $H_{1,1}(\T,\X)=H_{1,1,}(\T,X_0,X_1).$ From the definition of $F_{1,k}(\T,\X)$ given in \eqref{ddef}, we get
$$F_{1,k}(\T,\X)=\cF^0_{k-1}(\X)G_{1,1}(\T,X_0,X_1)+\cF^1_{k-1}(\X)H_{1,1}(\T,X_0,X_1)\in\langle X_0,X_1\rangle,
$$ and a choice for $D_X(F_{1,k}(\T,\X))$ is actually
\begin{equation}\label{ffru}
D_X(F_{1,k}(\T,\X))=\cF^0_{k-1}(\X)G_{1,1}(\T,p^0_2(\T),p^1_2(\T))+\cF^1_{k-1}(\X)H_{1,1}(\T,p^0_2(\T),p^1_2(\T))\end{equation}
From \eqref{form}, we actually get that  $F_{2,1}(\T,\X)\in\K[\T,X_0, X_1],$ and hence
$$F_{2,1}(\T,p^0_2(\T),p^1_2(\T))=0=T_0G_{1,1}(\T,p^0_2(\T),p^1_2(\T))+T_1H_{1,1}(\T,p^0_2(\T), p^1_2(\T)),
$$
so we conclude that there exist $q_2(\T)\in\K[\T]$ homogeneous of degree $2$ such that
$$\begin{array}{ccr}
G_{1,1}(\T,p^0_2(\T),p^1_2(\T))&=&T_1q_2(\T),\\
H_{1,1}(\T,p^0_2(\T), p^1_2(\T))&=&-T_0q_2(\T).
\end{array}
$$
Replacing the left hand side of the above identities in \eqref{ffru}, we get
$$D_X(F_{1,k}(\T,\X))=(T_1\cF^0_{k-1}(\X)-T_0\cF^1_{k-1}(\X))q_2(\T)\in\langle F_{1,k-1}(\T,\X)\rangle,
$$
which concludes the proof of the claim.
\end{proof}

\begin{lemma}\label{koko}
The set $$\{\cE_d(\X),\,F_{1,k-1}(\T,\X),\,F_{1,k}(\T,\X),\,F_{2,1}(\T,\X),\,F_{3,k-2}(\T,\X),\ldots, F_{2(k-2)-1,1}(\T,\X)\}$$ is contained in the ideal $\langle X_0, X_1\rangle.$
\end{lemma}
\begin{proof}
Each of the $F_{2(k-j)-1,j}(\T,\X)$ is actually equal to $D_\T(F_{2(k-j+1)-1,j-1}(\T,\X)),$ which by definition of this operator, its image always lies in $\langle X_0, X_1\rangle.$
\par The claim for $F_{2,1}(\T,\X)$ follows from its definition in \eqref{form}, and for $F_{1,k}(\T,\X)$ from Lemma \ref{k-1}. To conclude, due to \eqref{cucco}, we also have that 
$\cE_d(\X)\in\langle X_0,\,X_1\rangle.$
\end{proof}

Now we are ready for the main result of this section.
\begin{theorem}\label{mt1o}
Suppose $\mu=2,\,d=2k-1$ with $k\geq2$ and the parametrization $\phi$ induced by the data $\u(\T)$ being proper with a very singular point. Then, the following  $k+2=\frac{d+5}2$ polynomials form a minimal set of generators of $\kK:$
$${\mathrm F}_o:=\{\cE_d(\X),\,F_{2,1}(\T,\X),\, F_{2(k-1)-1,1}(\T,\X),\, \ldots ,\,  F_{1,k-1}(\T,\X),\, F_{1,k}(\T,\X)  \}.
$$
\end{theorem}

\begin{proof}
Theorem \ref{mtv} shows that the family ${\mathrm F}_o\setminus\{F_{1,k}(\T,\X)\}$ is a set of minimal generators of the ideal that generates it. Lemma \ref{loma} combined with the pseudo-homogeneity of the elements in this family, show that by adding $F_{1,k}(\T,\X)$ to the list, we still get a minimal set of generators (of the ideal generated by the whole family).
\par
Let us show now that ${\mathrm F}_o$ generates $\kK$. Due to Theorem \ref{mu}, it is enough to consider $G_{i,j}(\T,\X)\in\kK$ of bidegree $(i,j)$ with
$i< d-\mu.$ We will proceed by induction on $i$.
\begin{itemize}
\item If $i=0,$ as $\cE_d(\X)$ generates $\kK\cap\K[\X],$ the claim follows straightforwardly.
\item If $i=1$, by  Proposition \ref{multiple}, we have
$$G_{1,j}\big(\cF^0_{k-1}(\X), \cF^1_{k-1}(\X),\X\big)=\cE_d(\X)\,{\mathcal A}_{j-k}(\X),$$
with ${\mathcal A}_{j-k}(\X)\in\K[\X]_{j-k}.$ Then, it is easy to see that
$$\mbox{Res}_\T\big(G_{1,j}(\T,\X)-{\mathcal A}_{j-k}(\X)\,F_{1,k}(\T,\X),\,F_{1,k-1}(\T,\X))\big)=0
$$
by evaluating the first polynomial in the only zero of the second. But this implies that
$$G_{1,j}(\T,\X)-{\mathcal A}_{j-k}(\T,\X)\,F_{1,k}(\T,\X)\in\kK_{1,j}\cap\langle F_{1,k-1}(\T,\X)\rangle,$$
\item For $i=2$, we compute $\mbox{Res}_\T(G_{2,j}(\T,\X), F_{1,k-1}(\T,\X))$ to get $\cE_d(\X)\,{\mathcal A}_{j-1}(\X),$ with ${\mathcal A}_{j-1}(\X)\in\K[\X]_{j-1}.$
By reasoning as in the previous case, we get
$$G_{2,j}(\T,\X)-{\mathcal A}_{j-1}(\X)\,F_{2,1}(\T,\X)\in\kK_{2,j}\cap\langle F_{1,k-1}(\T,\X)\rangle,$$
as this polynomial also vanishes after the specialization $\T\mapsto\underline{\cF}_{k-1}(\X).$
\item If $i\geq3,$ then we can apply $\cD_T$ to $G_{i,j}(\T,\X)$ and get, by Proposition \ref{tec1}, $\cD_T(G_{i,j}(\T,\X))\in\kK_{i-2,j}.$ Now we use the inductive hypothesis and
get the following identity where all elements are polynomials in $\K[\T,\X]:$
\begin{equation}\label{pp}
\begin{array}{ccl}
\cD_T(G_{i,j}(\T,\X))&=&A(\T,\X)\cE_d(\X)+B(\T,\X)F_{1,k}(\T,\X)
+C(\T,\X)F_{2,1}(\T,\X)\\ &&+\sum_{1\leq 2(k-m)-1\leq i-2}D_m(\T,\X)F_{2(k-m)-1,m}(\T,\X).
\end{array}
\end{equation}
Due to \eqref{ident}, we have that $G_{i,j}(\T,\X)=\cD_X\big(\cD_T(G_{i,j}(\T,\X))\big)$ modulo $F_{2,1}(\T,\X),$ and thanks to Lemma \ref{koko}, we can apply $D_X(\cdot)$ to each of the members of the right hand side of \eqref{pp}. We verify straightforwardly from the definition given in \eqref{orozco} that $$\cD_X(F_{2(k-m)-1,m}(\T,\X))= F_{2(k-m+1)-1,m-1}(\T,\X),$$  and then get the following identity modulo  $F_{2,1}(\T,\X):$
$$
\begin{array}{ccl}
G_{i,j}(\T,\X)&=&A(\T,\X)\cD_X(\cE_d(\X))+B(\T,\X)\cD_X(F_{1,k}(\T,\X))\\
&&+C(\T,\X)\cD_X(F_{2,1}(\T,\X))+\sum_{1\leq 2(k-m)-1\leq i-2}D_m(\T,\X)\cD_X(F_{2(k-m)-1,m}(\T,\X)) \\ \\
&=&A(\T,\X)\cD_X(\cE_d(\X))+\tilde{B}(\T,\X)F_{1,k-1}(\T,\X)+\\ &&\sum_{1\leq 2(k-m)-1\leq i-2}D_m(\T,\X)F_{2(k-m+1)-1,m-1}(\T,\X),
\end{array}
$$
where the last equality holds thanks to \eqref{kk-1}. The claim now follows straightforwardly from this identity by noting that $\cD_X(\cE_d(\X))\in\kK_{2,d-1},$ and that we just proved (this is the case $i=2$) that this part of $\kK$ is generated by elements of ${\mathrm F}_o.$ This concludes the proof.
\end{itemize}
\end{proof}

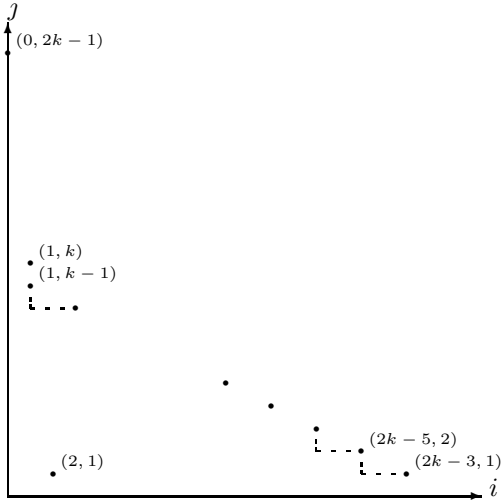
\setlength{\unitlength}{1mm}
\begin{picture}(68,68)
\put(0,0){\vector(1,0){63}}
 \put(64,0){$i$}
 \put(0,0){\vector(0,1){63}}
 \put(0,64){$j$}
\put(6,3){\circle*{0.7} }
\put(7,4){\tiny{$(2,1)$}}
\put(53,3){\circle*{0.7} }
 \put(54,4){\tiny{$(2k-3,1)$}}
\put(47,6){\circle*{0.7} }
 \put(48,7){\tiny{$(2k-5,2)$}}
 \put(41,9){\circle*{0.7} }
 \put(35,12){\circle*{0.7} }
  \put(29,15){\circle*{0.7} }
  \multiput(47,3)(2,0){3}{\line(1,0){0.5}}
  \multiput(47,3)(0,1){2}{\line(0,1){0.5}}
 \multiput(41,6)(2,0){3}{\line(1,0){0.5}}
  \multiput(41,6)(0,1){2}{\line(0,1){0.5}}
 \put(3,28){\circle*{0.7}}
\put(4,29){\tiny{$(1,k-1)$}}
\put(3,31){\circle*{0.7}}
\put(4,32){\tiny{$(1,k)$}}
\put(9,25){\circle*{0.7}}
\multiput(3,25)(2,0){3}{\line(1,0){0.5}}
  \multiput(3,25)(0,1){2}{\line(0,1){0.5}}
\put(0,59){\circle*{0.7}}
\put(1,60){\tiny{$(0,2k-1)$}}
\end{picture}

\begin{example}
For $k\geq 3$, consider
\[u_0(T_0,T_1)=T_0^{2k-1},\, u_1(T_0,T_1)=T_0^{2k-3}T_1^2,\, u_2(T_0,T_1)=T_1^{2k-1}.\]
 These polynomials parametrize a curve of degree $2k-1$ with $\mu =2$ and 
 \[T_1^2X_0-T_0^ 2X_1,\qquad  T_1^{2k-3}X_1-T_0^{2k-3}X_2\] as $\mu $ basis. The minimal system of generators of $ \kK$ given in Theorem \ref{mt1o} is in this case
\[\begin{array}{l}
\cE(\X)=X_1^{2k-1}-X_0^{2k-3}X_2^2,\\
F_{2,1}(\T,\X)=T_1^2X_0-T_0^ 2X_1,\\
F_{d-2,1}(\T,\X)=F_{2(k-1)-1,1}=T_1^{2k-3}X_1-T_0^{2k-3}X_2,\\
\quad \vdots \\
F_{2(k-j)-1,j}(\T,\X)=T_1^{2(k-j)-1}X_1^j-T_0^{2(k-j)-1}X_0^{j-1}X_2\\
 \quad \vdots \\
F_{1,k-1}(\T,\X)=T_1X_1^{k-1}-T_0X_0^{k-2}X_2,\\
F_{1,k}(\T,\X)=T_0X_1^k-T_1X_0^{k-1}X_2.
\end{array}
\]
\end{example}

\bigskip
\subsection{$d$ even}
We will assume here that $d=2k,$ with $k\geq3$ and that $\mu=2.$
In this case, the family in Theorem \ref{mtv}(4) explicits as
$$\{\cE_d(\X),\,F_{2,1}(\T, \X),\,F_{2,k-1}(\T, \X),\,F_{4, k-2}(\T, \X),\,\ldots, F_{2(k-1),1}(\T, \X)\},$$
and note that there are not generators of degree $1$ in $\T$. We will produce two of them by making suitable polynomial combinations among $F_{2,1}(\T,\X)$ and
$F_{2,k-1}(\T,\X)$ as follows:
Write
\begin{equation}\label{suicigia}
\begin{array}{ccl}
F_{2,1}(\T,\X)&=&T_0^2\cF^0_1(\X)+T_1^2\cF^1_1(\X)+T_0T_1\cF^{*}_1(\X)\\
F_{2,k-1}(\T,\X)&=&T_0^2\cM^0_{k-1}(\X)+T_1^2\cM^1_{k-1}(\X)+T_0T_1\cM^{*}_{k-1}(\X),
\end{array}
\end{equation}
and define $F^0_{1,k}(\T,\X)$ and $F^1_{1,k}(\T,\X)$ via the following identities:
\begin{equation}\label{form1}
\begin{array}{ccc}
\cM^0_{k-1}(\X)F_{2,1}(\T,\X)-\cF^0_1(\X)F_{2,k-1}(\T,\X)&=&T_1\,F^0_{1,k}(\T,\X),\\
\cM^1_{k-1}(\X)F_{2,1}(\T,\X)-\cF^1_1(\X)F_{2,k-1}(\T,\X)&=&T_0\,F^1_{1,k}(\T,\X).
\end{array}
\end{equation}
We write
\begin{equation}\label{cochu}
\begin{array}{ccl}
F^0_{1,k}(\T,\X)&=&T_0\cF^{0,0}_k(\X)-T_{1}\cF^{0,1}_k(\X)\\
F^1_{1,k}(\T,\X)&=&T_0\cF^{1,0}_k(\X)-T_{1}\cF^{1,1}_k(\X)
\end{array}
\end{equation}

\begin{proposition}\label{1forms}
$^{}$
\begin{enumerate}
\item $F^i_{1,k}(\T,\X)\in\kK_{1,k}\cap\langle X_0,\,X_1\rangle,$ for $i=0,1.$
\item Up to a nonzero constant in $\K,$
$$\cF^{0,0}_k(\X)\cF^{1,1}_k(\X)-\cF^{1,0}_k(\X)\cF^{0,1}_k(\X)=\mbox{Res}_\T(F^0_{1,k}(\T,\X), F^1_{1,k}(\T,\X))=\cE_d(\X).$$
\item $\{F^0_{1,k}(\T,\X),\,F^1_{1,k}(\T,\X)\}$ is a basis of the $\K[\X]$-module $\kK_{1,*}.$
\item Modulo $F_{2,1}(\T,\X),\, D_X\big(F^i_{1,k}(\T,\X)\big)\in\langle F_{2,k-1}(\T,\X)\rangle$ for $i=0,1.$
\end{enumerate}
\end{proposition}
\begin{proof} $^{}$
\begin{enumerate}
\item Follows straightforwardly from the definition of $F^i_{1,k}(\T,\X)$ given in \eqref{form1}, by taking into account that both $F_{2,1}(\T,\X)$ and $F_{2,k-1}(\T,\X)$ are elements of
$\kK\cap\langle X_0, X_1\rangle.$
\item The fact that $\mbox{Res}_\T(F^0_{1,k}(\T,\X), F^1_{1,k}(\T,\X))$ coincides with $\cF^{0,0}_k(\X)\cF^{1,1}_k(\X)-\cF^{1,0}_k(\X)\cF^{0,1}_k(\X)$ follows just from the definition of $\mbox{Res}_\T$ and \eqref{cochu}. As both $F^i_{1,k}(\T,\X)\in\kK,\, i=0,1,$ it  turns out then that  $\mbox{Res}_\T(F^0_{1,k}(\T,\X), F^1_{1,k}(\T,\X))$ must be a multiple of $\cE_d(\X).$ Computing degrees, both polynomials have the same degree $2k=d,$ then the resultant actually must be equal to $\lambda\,\cE_d(\X).$ To see that $\lambda\neq0,$ it is enough to show that the forms $F^i_{1,k}(\T,\X)$ are
$\K$-linearly independent as they have the same bidegree. Suppose that this is not the case, and write $\lambda_0F^0_{1,k}(\T,\X)+\lambda_1F^1_{1,k}(\T,\X)=0$ with $\lambda_0,\,\lambda_1\in\K,$ not both of them equal to zero. We will have then, from \eqref{form1}:
$$(\lambda_0T_0\cM^0_{k-1}(\X)+\lambda_1T_1\cM^1_{k-1}(\X))F_{2,1}(\T,X)=(\lambda_0T_0\cF^0_1(\X)+\lambda_1T_1\cF^1_1(\X))F_{2,k-1}(\T,\X)
$$
From Theorem \ref{mtv} (2), we know that $F_{2,1}(\T,\X)$ and $F_{2,k-1}(\T,\X)$ are coprime, so an identity like above cannot hold unless it is identically zero, which forces $\lambda_0=\lambda_1=0$,  a contradiction to our assumption.
\item The $\K[\X]$-linear independence of the family  $\{F^0_{1,k}(\T,\X),\,F^1_{1,k}(\T,\X)\}$ follows from the fact that their $\T$-resultant is not zero, which has been shown already in (2). So, it is enough to show that any other element in $\kK_{1,*}$ is a polynomial combination of these two. Let $G_{1,j}(\T,\X)\in\kK_{1,j}.$ Then, as before, we have that
$$\mbox{Res}_\T\big(F^0_{1,k}(\T,\X),G_{1,j}(\T,\X)\big)=\cE_d(\X)\cP_{j-k}(\X),
$$ with $\cP_{j-k}(\X)\in\K[\X]_{j-k}.$ If the latter is identically zero, then the claim follows straightforwardly. Otherwise (note that this immediately implies $j\geq k$), set
$$\tilde{G}_{1,j}(\T,\X):=G_{1,j}(\T,\X)-\cP_{j-k}(\X)\,F^1_{1,k}(\T,\X)\in\K[\T,\X]_{1,j}.
$$
It is then easy to show that $\mbox{Res}_\T\big(F^0_{1,k}(\T,\X),\tilde{G}_{1,j}(\T,\X)\big)=0,$ which implies that $\tilde{G}_{1,j}(\T,\X)\in\langle F^1_{1,k}(\T,\X)\rangle$, and so we get immediately from the definition of $\tilde{G}_{1,j}(\T,\X)$ given above that $G_{1,j}(\T,\X)\in\langle F^0_{1,k}(\T,\X), F^1_{1,k}(\T,\X)\rangle.$
\item First note that, because of what we just proved in (1), the operator $\cD_X$ can be applied to $F^i_{1,k}(\T,\X)$ for $i=0,1.$ Also, it is immediate to check that the polynomials
$\cF^0_1(\X)$ and $\cF^1_1(\X)$ defined in \eqref{suicigia}  belong to $\langle X_0,X_1\rangle.$ So we can actually apply $\cD_X$ to both identities in \eqref{form1} and define
$\cD_X(F^i_{1,k}(\T,\X))$ in such a way that
$$
\begin{array}{ccc}
-\cF^0_1(p^0_2(\T), p^1_2(\T))F_{2,k-1}(\T,\X)&=&T_1\cD_X(F^0_{1,k}(\T,\X)),\\
-\cF^1_1(p^-_2(\T), p^1_2(\T))F_{2,k-1}(\T,\X)&=&T_0\cD_X(F^1_{1,k}(\T,\X)).
\end{array}
$$
Note that $F_{2,k-1}(\T,\X)$ cannot not have any proper factor. Indeed, by Theorem \ref{mtv}, it belongs to a subset of a minimal generator of the (prime) ideal $\kK.$ This shows that
$T_i$ divides  $-\cF^i_1(p^0_2(\T), p^1_2(\T))$ for $i=0,1$ and hence $\cD_X(F^i_{1,k}(\T,\X))\in\langle F_{2,k-1}(\T,\X)\rangle$ for $i=0,\,1.$
\end{enumerate}
\end{proof}

Now we are ready to prove the main theorem of this section. Note just that if $n=4$ and $\mu=2,$ if there is a point of multiplicity strictly larger than $\mu$, then it is a triple point and that forces $\mu=1,$ a contradiction with our hypothesis.

\begin{theorem}\label{mt1e}
Suppose $\mu=2,\,d=2k$ with $k\geq3$ and the parametrization being proper with a very singular point. Then, a minimal set of generators of $\kK$ is the following set of $k+3=\frac{d+6}2$ polynomials
$${\mathrm F}_e:=\{\cE_d(\X),\,F^0_{1,k}(\T,\X),\,F^1_{1,k}(\T,\X),\,F_{2,1}(\T,\X),\,F_{2,k-1}(\T,\X),\ldots, F_{2(k-1),1}(\T,\X)\}.
$$
\end{theorem}

\begin{proof}
The proof follows the same lines as the proof of Theorem \ref{mt1o}. To begin with, Theorem \ref{mtv} combined with Proposition \ref{1forms}(3) show that ${\mathrm F}_e$ is a minimal set of generators of an ideal contained in $\kK.$ In order to see that they are equal, we will proceed again by induction on the $\T$-degree of the
forms, the case $i=0$ follows analogously from the proof of Theorem \ref{mt1o}. For $i=1,$ the claim has been proven in Proposition \ref{1forms}(3).
\par Suppose then $i=2$, and write $G_{2,j}\in\kK_{2,j}$ as
$$G_{2,j}(\T,\X)=T_0^2\cG^0_j(\X)+T_1^2\cG^1_j(\X)+T_0T_1\cG^{*}_j(\X),
$$
Recall the notation we introduced in \eqref{suicigia} and write
$$\begin{array}{ccl}
\cG^0_j(\X) F_{2,1}(\T,\X)-\cF^0_1(\X) G_{2,j}(\T,\X)=T_1\,H_{1,j+1}(\T,\X),\\
\cG^0_j(\X) F_{2,k-1}(\T,\X)-\cM^0_{k-1}(\X) G_{2,j}(\T,\X)=T_1\,H^*_{1,j+k-1}(\T,\X),
\end{array}
$$
so we get
\begin{equation}\label{soupy}
\cM^0_{k-1}(\X)\cG^0_j(\X)F_{2,1}(\T,\X)-\cF^0_1(\X)\cG^0_j(\X) F_{2,k-1}(\T,\X)=T_1\, H^{**}_{1,j+k}(\T,\X)
\end{equation}
with $H_{1,j+1}(\T,\X),\,H^*_{1,j+k-1}(\T,\X),\,H^{**}_{1,j+1}(\T,\X)\in\kK_{1,*}$. By Proposition \ref{1forms}(3), we know that $\kK_{1,*}$ is generated by $\langle F^0_{1,k}(\T,\X),\,F^1_{1,k}(\T,\X)\rangle,$ so we have
$$\begin{array}{lcl}
H_{1,j+1}(\T,\X)&=&\alpha_{j-k+1}(\X)F^0_{1,k}(\T,\X)+\beta_{j-k+1}(\X)F^1_{1,k}(\T,\X),\\
H^{*}_{1,j+k-1}(\T,\X)&=&\alpha^*_{j-1}(\X)F^0_{1,k}(\T,\X)+\beta_{j-1}(\X)F^1_{1,k}(\T,\X)\\
H^{**}_{1,j+k}(\T,\X)&=&\alpha^{**}_{j}(\X)F^0_{1,k}(\T,\X)+\beta^{**}_{j}(\X)F^1_{1,k}(\T,\X).
\end{array}
$$
Note that
\begin{equation}\label{pua}
\alpha^{**}_{j}(\X)=\cM^0_{k-1}(\X)\,\alpha_{j-k+1}(\X)-\cF^0_1(\X)\,\alpha^*_{j-1}(\X).
\end{equation}
From \eqref{form1}, we deduce
$$\cG^0_j(\X)(\cM^0_{k-1}(\X)F_{2,1}(\T,\X)-\cF^0_1(\X)F_{2,k-1}(\T,\X))=T_1\,\cG^0_j(\X)F^0_{1,k}(\T,\X).$$
By substracting this identity from \eqref{soupy}, and using the obvious fact that $F^0_{1,k}(\T,\X)$ and $F^1_{1,k}(\T,\X)$ are $\K[\X]$-linearly independent, we deduce that
\begin{equation}\label{soso}
\cG^0_j(\X)=\alpha^{**}_{j}(\X)=\cM^0_{k-1}(\X)\,\alpha_{j-k+1}(\X)-\cF^0_1(\X)\,\alpha^*_{j-1}(\X),
\end{equation}
the last equality is \eqref{pua}.
So, by setting $$\tilde{G}_{2,j}(\T,\X):=G_{2,j}(\T,\X)-\alpha_{j-k+1}(\X)F_{2,k-1}(\T,\X)+\alpha^*_{j-1}(\X) F_{2,1}(\T,\X),$$
due to \eqref{soso}, we easily deduce that $\tilde{G}_{2,j}=T_1G^*_{1,j}(\T,\X)$, with $G^*_{1,j}(\T,\X)\in\kK_{1,j}.$ Again by Proposition \ref{1forms}(3), it turns out that
$G^*_{1,j}(\T,\X)\in\langle F^0_{1,k}(\T,\X),\,F^1_{1,k}(\T,\X)\rangle$ and hence $G_{2,j}(\T,\X)\in\langle F^0_{1,k}(\T,\X),\,F^1_{1,k}(\T,\X),\,F_{2,1}(\T,\X),\,F_{2,k-1}(\T,\X)\rangle$
which proves the claim for $i=2.$
\par\smallskip
If $i\geq2$ we proceed exactly as in the proof of Theorem \ref{mt1o}, and we only have to verify that $\cD_X(F^0_{1,k}(\T,\X))$ and $\cD_X(F^1_{1,k}(\T,\X)$ belong to the ideal generated by ${\mathrm F}_e.$ But this follows immediately from Proposition \ref{1forms} (4). This completes the proof of the Theorem
\end{proof}

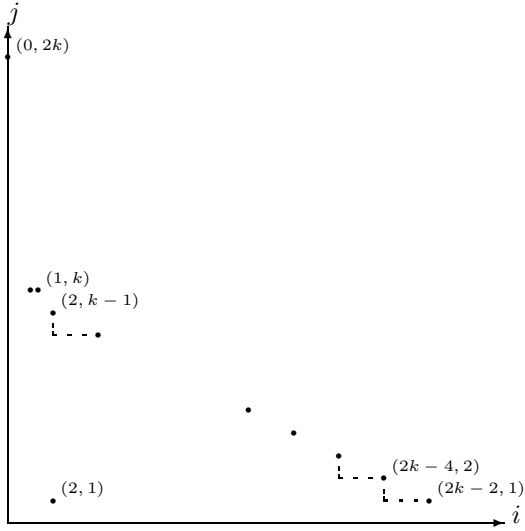
\setlength{\unitlength}{1mm}
\begin{picture}(68,68)
\put(0,0){\vector(1,0){66}}
 \put(67,0){$i$}
 \put(0,0){\vector(0,1){66}}
 \put(0,67){$j$}
\put(6,3){\circle*{0.7} }
\put(7,4){\tiny{$(2,1)$}}

\put(56,3){\circle*{0.7} }
 \put(57,4){\tiny{$(2k-2,1)$}}
 \put(50,6){\circle*{0.7} }
 \put(51,7){\tiny{$(2k-4,2)$}}
  \put(44,9){\circle*{0.7} }
 \put(38,12){\circle*{0.7} }
  \put(32,15){\circle*{0.7} }
 \multiput(50,3)(2,0){3}{\line(1,0){0.5}}
 \multiput(50,3)(0,1){2}{\line(0,1){0.5}}
 \multiput(44,6)(2,0){3}{\line(1,0){0.5}}
 \multiput(44,6)(0,1){2}{\line(0,1){0.5}}

 \put(6,28){\circle*{0.7}}
\put(7,29){\tiny{$(2,k-1)$}}

\put(3,31){\circle*{0.7}}
\put(4,31){\circle*{0.7}}
\put(5,32){\tiny{$(1,k)$}}

\put(12,25){\circle*{0.7}}
\multiput(6,25)(2,0){3}{\line(1,0){0.5}}
\multiput(6,25)(0,1){2}{\line(0,1){0.5}}
\put(0,62){\circle*{0.7}}
\put(1,63){\tiny{$(0,2k)$}}
\end{picture}

\begin{example}
For $k\geq 3$, consider
\[u_0(T_0,T_1)=T_0^{2k},\, u_1(T_0,T_1)=T_0^{2k-2}(T_1^2+T_0T_1),\, u_2(T_0,T_1)=T_1^{2k-2}(T_1^2+T_0T_1).\]
 These polynomials parametrize properly a curve of degree $2k$ with $\mu =2$ and
 \[ (T_1^2+T_oT_1)X_0-T_0^ 2X_1,\qquad  T_1^{2k-2}X_1-T_0^{2k-2}X_2\] as $\mu $ basis. Indeed, by computing the implicit equation, we get
$$\cE_{2k}(\X)=X_1^{2k}-\frac{1}{2^{2k-3}}\left(\sum_{j=0}^{k-1}{2k-2\choose 2j}X_0^{2k-2j-2}(X_0^2+4X_0X_1)^j\right)X_1X_2+X_0^{2k-2}X_2^2.
$$  
\end{example}

\bigskip

\bigskip
\section{Adjoints}\label{adjoints}
We now turn our attention to geometric features of elements in $\kK_{1,*}.$  Recall that a curve $\tilde{\C}$ is adjoint to $\C$ if for any point ${\bf p}\in\C$, including ``virtual points'', we have
\begin{equation}\label{isi}
m_{\bf p}(\tilde{\C})\geq m_{\bf p}(\C)-1.
\end{equation}
Here, $m_p(\C)$ denotes the multiplicity of ${\bf p}$ with respect to $\C$.  Adjoint curves are of importance in computational algebra due to their use in the inverse of the implicitization problem, i.e. the so-called ``parametrization problem'', see \cite{SWP08} and the references therein. For a more geometric approach to adjoints, we refer the reader to \cite{cas00}.
\begin{definition}
A pencil of adjoints of  $\C$ of degree $\ell\in\N$ is an element $T_0{\mathcal C}^0_\ell(\X)+T_1{\mathcal C}^1_\ell(\X)\in\K[\T,\X],$ with ${\mathcal C}^i_\ell(\X)$ of degree $\ell$, defining a curve adjoint of $\C$, for $i=0,1.$ 
\end{definition}
For $\ell\in\Z_{\geq0},$ we denote with $\mbox{Adj}_\ell(\C)$ the $\K$-vector space of pencils of adjoints of $\C$ of degree $\ell$.
In \cite[Corollary $4.10$]{bus09}, it is shown that  if $\C$ has $\mu=2$ and only mild singularities, then both $\kK_{1,d-2}$ and $\kK_{1,d-1}$ are contained in $\mbox{Adj}_\ell(\C), \ \ell=d-2,\,d-1$ respectively.  We will show here that if $\C$ has $\mu=2$ and a very singular point, then $\mbox{Adj}_\ell(\C)\cap\kK_{1,\ell}$ is strictly contained in $\kK_{1,\ell}$ if
the later is not zero. We will also compute the dimension of these finite dimensional $\K$-vector spaces for a generic $\C$ to  measure the difference between them. 
\begin{lemma}\label{emma}
With the notation introduced in the previous section, for $i=k-1,\,k$ and $j=0,\,1,$ we have that
$$\begin{array}{l}
F_{1,i}(\T,\X)\in\langle X_0,X_1\rangle^{i-1}\setminus\langle X_0,\,X_1\rangle^i,\\
F^j_{1,k}(\T,\X),\in\langle X_0,X_1\rangle^{k-1}\setminus\langle X_0,\,X_1\rangle^k.
\end{array}$$
\end{lemma}
\begin{proof}
 The operator ${\mathcal D}_T$ from Definition \ref{DT}, when applied to a polynomial in $\langle X_0,\,X_1\rangle^\ell$, has its image in $\langle X_0,\,X_1\rangle^{\ell+1}.$ From here, it is easy to deduce that $F_{1,k-1}(\T,\X)\in\langle X_0,\, X_1\rangle^{k-2}.$ If it actually belonged to
$\langle X_0,\,X_1\rangle^{k-1},$ then it would not depend on $X_2.$ But as
$$\mbox{Res}_\T\big(F_{2,1}(\T,\X),\,F_{1,k-1}(\T,\X)\big)=\cE_d(\X)$$
and $F_{2,1}(\T,\X)$ does not depend on $X_2,$ we would then have that $\cE_d(\X)\in\K[X_0, X_1]$, which is a contradiction with the irreducibility of this polynomial. The same argument holds for $F_{1,k}(\T,\X)$ by noting now that
$$\mbox{Res}_\T\big(F_{2,1}(\T,\X),\,F_{1,k}(\T,\X)\big)=\cE_d(\X)\,{\mathcal A}_2(\X),$$
with ${\mathcal A}_2(\X)\neq0.$
\par 
For the second part of the proof, we get that $F^j_{1,k}(\T,\X),\in\langle X_0,X_1\rangle^{k-1}$  for $j=0,\,1$ straightforwardly from the definition of these forms given in \eqref{form1}.  An explicit computation shows that also
$$\mbox{Res}_\t\big(F^j_{1,k}(\T,\X),\,F_{2,1}(\T, \X)\big)=\pm\cE_d(\X)\,{\mathcal L}^j_1(\X)
$$ with ${\mathcal L}^j_1(\X)\neq0,$ which proves that $F_{1,k}(\T,\X)$ has term which is linear in $X_2$.
\end{proof}
\smallskip
In the sequel, we set ${a\choose b}=0$ if $a<b.$ For a $\K[\X]$-graded module $M$ and an integer $\ell,$  we denote with $M_\ell$ the $\ell$-th graded piece of $M$. 
\begin{proposition}\label{topos} Let $\phi$ as in \eqref{param} be a proper parametrization of a curve $\C$ having $\mu=2$ and a very singular point. For $\ell\geq0,$
\begin{enumerate}
\item if $d=2k-1,$ then $\kK_{1,\ell}=\langle F_{1,k-1}(\T,\X)\rangle_\ell\oplus\langle F_{1,k}(\T,\X)\rangle_\ell$ and the dimension of this $\K$-vector space is ${\ell-k+3\choose 2}+{\ell-k+2\choose 2}.$ 
\item If $d=2k,$ then $\kK_{1,\ell}=\langle F^0_{1,k}(\T,\X)\rangle_\ell\oplus\langle F^1_{1,k}(\T,\X)\rangle_\ell,$ its $\K$-dimension being $2{\ell-k+2\choose 2}.$
\end{enumerate}
\end{proposition}
\begin{proof}
Suppose first $d=2k-1$.
From the statement of Theorem \ref{mt1o}, we have that $\kK_{1,*}=\langle F_{1,k-1}(\T,\X),\,F_{1,k}(\T,\X)\rangle_{\K[\X]}.$ Moreover, from Lemma \ref{loma} and the proof of Theorem \ref{mt1o}, we easily deduce that
$$\langle F_{1,k-1}(\T,\X),\,F_{1,k}(\T,\X)\rangle_\ell=\langle F_{1,k-1}(\T,\X)\rangle_\ell\oplus\langle F_{1,k}(\T,\X)\rangle_\ell
$$
for any $\ell\geq0.$ From here, the claim follows straightforwardly by computing dimensions in each of the subspaces involved in the last equality. The case $d=2k$ follows analogously, using now Proposition \ref{1forms} (3).
\end{proof}
\begin{theorem}\label{ttt}
Let $\phi$ as in \eqref{param} be a proper parametrization of a curve $\C$ having $\mu=2$ and a very singular point. For any $\ell\geq0,$ 
\begin{itemize}
\item If $d=2k-1,$ then 
$$\dim_\K\left(\mbox{Adj}_\ell(\C)\cap\kK_{1,\ell}\right)\leq \left\{\begin{array}{cl}
0&\,\mbox{if}\ \ell<2k-3,\\
\ell(\ell-2k+4)& otherwise
\end{array}
\right.$$
\item If $d=2k,$ then $$\dim_\K\left(\mbox{Adj}_\ell(\C)\cap\kK_{1,\ell}\right)\leq \left\{\begin{array}{cl}
0&\,\mbox{if}\ \ell<2k-2,\\
\ell(\ell-2k+3)& otherwise
\end{array}
\right.$$
\end{itemize}
For a generic curve $\C$ with $\mu=2$ and a very singular point, the equality actually holds.
\end{theorem}

\begin{proof}
Suppose $d=2k-1$ with $k\geq3$ (otherwise there cannot be a point of multiplicity larger than $2$), and w.l.o.g. assume that $(0:0:1)$ is the point of multiplicity $d-2=2k-3$. 
Fix $\ell\geq0,$ and set
$${\mathfrak Z}_\ell=\langle x_0,\,x_1\rangle^{d-3}\cap\kK_{1,\ell}.$$
Due to \eqref{isi} applied to $p=(0:0:1)$, it turns out that $\mbox{Adj}_\ell(\C)\cap\kK_{1,\ell}\subset {\mathfrak Z}_\ell.$ Moreover, the equality holds for a generic curve with $\mu=2$ and $(0:0:1)$ being very singular. Indeed, such a curve has all its singularities of ordinary type (i.e. there are no ``virtual points''). For this class of curves it is easy to show that any nonzero element in ${\mathfrak Z}_\ell$ is a pencil of adjoints, as we already know that $(0:0:1)$ has the correct multiplicity, plus the fact that all the other singular points have multiplicity two thanks to Proposition \ref{kkk} (and are ordinary due to genericity). So, condition \eqref{isi} for these points is satisfied provided that the pencil vanish also at these points, and this follows from Proposition \ref{ssing}.

To compute the dimension of ${\mathfrak Z}_\ell$, Proposition \ref{topos} and Lemma \ref{emma}, implies that the set
$\{\X^{\underline\alpha}F_{1,k-1}(\T,\X),\,\X^{\underline\beta}F_{1,k}(\T,\X)\}$
with $|\underline\alpha|=\ell-k+1,\,\alpha_0+\alpha_1\geq k-2,\,
|\underline\beta|=\ell-k,\,\beta_0+\beta_1\geq k-3,$ is a basis of ${\mathfrak Z}_\ell.$ If $\ell<2k-3,$ the cardinality of this set is zero. Otherwise, it is  equal to
$$\sum_{j=k-2}^{\ell-k+1}(j+1)+\sum_{j=k-3}^{\ell-k}(j+1)=\ell(\ell-2k+4).
$$
The proof for $d=2k$ follows mutatis mutandis the argument above.
\end{proof}
\begin{remark}
Combining the dimensions computed in Proposition \ref{topos} and Theorem \ref{ttt}, we get that 
$$\dim\left(\kK_{1,\ell} / \mbox{Adj}_\ell(\C)\cap\kK_{1,\ell}\right)\geq\left\{\begin{array}{ll}
(k-2)^2&\,\mbox{if}\ d=2k-1\\
(k-1)(k-2)&\,\mbox{if}\ d=2k,
\end{array}
\right.
$$
with equality for $\ell\geq d-2$ and $\C$ generic in this family of curves. Note that the dimension of the quotient is independent of $\ell$ for $\ell\geq d-2.$
\end{remark}

\bigskip
\section{Curves with mild multiplicities}\label{nosing}
Now we turn to the case where there are no multiple points of multiplicity larger than $2$. In this case, a whole set of generators of $\kK$ has been given in \cite[Proposition $3.2$]{bus09}, and our contribution will be to show that this set is essentially minimal in the sense that there is only one element which can be removed from the list.

We start by recalling the construction of Bus\'e's generators. In order to do this, some tools from classical elimination theory of polynomials will be needed. As in the beginning, our $\mu$-basis will be supposed to be a fixed set of polynomials $\{P_{2,1}(\T,\X),\,Q_{d-2,1}(\T,\X)\}.$ Recall that in this situation, we now have
\begin{equation}\label{kko}
P_{2,1}(\T,\X)=T_0^2L^0_1(\X)+T_1^2L^1_1(\X)+T_0T_1L^*_1(\X),
\end{equation}
with $V_{\P^2}(L^0_1(\X),\,L^1_1(\X),\,L^*_1(\X))=\emptyset,$ in contrast with the previous case where this variety was the unique point in $\C$ having  multiplicity $d-2$ on the curve.

\subsection{Sylvester forms}
For $\bv=(v_0,\,v_1)\in\{(0,0),\,(1,0),\,(0,1)\},$ write
$$
\begin{array}{ccl}
P_{2,1}(\T,\X)&=&T_0^{1+v_0}P^{0,\bv}_{1-v_0,1}(\T,\X)+T_1^{1+v_1}P^{1,\bv}_{1-v_1,1}(\T,\X),\\
Q_{d-2,1}(\T,\X)&=&T_0^{1+v_0}Q^{0,\bv}_{d-3-v_0,1}(\T,\X)+T_1^{1+v_1}Q^{1,\bv}_{d-3-v_1,1}(\T,\X),
\end{array}
$$
and set
\begin{equation}\label{ssilv}
\Delta^\bv(\T,\X):=\left|
\begin{array}{cc}
P^{0,\bv}_{1-v_0,1}(\T,\X)&P^{1,\bv}_{1-v_1,1}(\T,\X)\\
Q^{0,\bv}_{d-3-v_0,1}(\T,\X)& Q^{1,\bv}_{d-3-v_1,1}(\T,\X)
\end{array}
\right|\in\K[\T,\X]_{d-2-|\bv|,2}.
\end{equation}
It is easy to see (see also \cite{bus09}) that these polynomials are elements of $\kK$,  well defined modulo $\kK_{*,1}.$ Note also that one has the following equality modulo $\kK_{*,1}$:
\begin{equation}\label{descarte}
\Delta^{(0,0)}(\T,\X)=T_0\Delta^{(1,0)}(\T,\X)=T_1\Delta^{(0,1)}(\T,\X),
\end{equation}
which essentially shows that these elements are not independent modulo $\kK$. These forms are called {\em Sylvester forms} in the literature, see for instance \cite{jou97, CHW08, bus09}.

\medskip
\subsection{Morley forms}
Now we will define more elements of $\kK$ of the form $\Delta_\bv(\T,\X),$ for $2\leq|\bv|\leq d-1.$ In order to do that, we first have to compute the {\em Morley form} of the polynomials $P_{2,1}(\T,\X)$ and $Q_{d-2,1}(\T,\X)$, as defined in \cite{jou97, bus09}, as follows: introduce a new set of variables $\S=S_0,\,S_1,$write
\begin{equation}\label{dolo}
\begin{array}{rcc}
P_{2,1}(\T,\X)-P_{2,1}(\S,\X)&=&P^0(\S,\T,\X)(T_0-S_0)+P^1(\S,\T,\X)(T_1-S_1)\\
Q_{d-2,1}(\T,\X)-Q_{d-2,1}(\S,\X)&=&Q^0(\S,\T,\X)(T_0-S_0)+Q^1(\S,\T,\X)(T_1-S_1)
\end{array}\, ,
\end{equation}
and define the {\em Morley form} of $P_{2,1}(\T,\X)$ and $Q_{d-2,1}(\T,\X)$ as
$$\mbox{Mor}(\S,\T,\X):=\left|
\begin{array}{cc}
P^0(\S,\T,\X)&P^{1}(\S,\T,\X)\\
Q^{0}(\S,\T,\X)& Q^{1}(\S,\T,\X)
\end{array}
\right|.
$$
Due to homogeneities, it is easy to see that we have the following monomial expansion of the Morley form:
\begin{equation}\label{morr}
\mbox{Mor}(\S,\T,\X)=\sum_{|\bv|\leq d-2} F^\bv_{d-2-|\bv|,2}(\T,\X)\,\S^\bv,
\end{equation}
with $F^{\bv}_{d-2-|\bv|,2}(\T,\X)\in\K[\T,\X]_{(d-2-|\bv|,2)}.$ It can also be shown (see for instance \cite{jou97} or \cite{bus09}) that the elements $F^{\bv}_{d-2-|\bv|,2}(\T,\X)$ are well defined modulo the ideal generated by $P_{2,1}(\T,\X)-P_{2,1}(\S,\X)$ and $Q_{d-2,1}(\T,\X)-Q_{d-2,1}(\S,\X).$

To define nontrivial elements in $\kK,$ we proceed as in \cite[Section 2.3]{bus09}:  Fix $i,\,1\leq i\leq d-2$ and let $\bM_i$ be the $(d-1-i)\times (d-2-i)$ matrix defined as follows
$$
\bM_i=
\left(\begin{array}{ccccc}
L^0_1(\X)&0&0&\ldots& F^{(d-2-i,0)}_{i,2}(\T,\X)\\
 L^*_1(\X)& L^0_1(\X)& 0&\ldots& F^{(d-3-i,1)}_{i,2}(\T,\X)\\
  L^1_1(\X)&L^*_1(\X)& L^0_1(\X)&\ldots& F^{(d-4-i,2)}_{i,2}(\T,\X)\\
\vdots&\ddots & \ddots& \dots &\vdots \\ \\
0&\ldots && L^1_1(\X)& F^{(0,d-2-i)}_{i,2}(\T,\X)
\end{array}
\right).
$$
By looking at the last column, we see that the rows of $\bM_i$ are indexed by monomials $\bv$ such that $|\bv|=d-2-i.$ For each of these monomials, we define
$\Delta^\bv_{i,d-1-i}(\T,\X)$ as the signed maximal minor of $\bM_i$ obtained by eliminating from this matrix the row indexed by $\bv.$ By looking at the homogeneities of
the columns of $\bM_i$, we easily get that $\Delta^\bv_{i,d-1-i}(\T,\X)\in\K[\T,\X]_{i,d-1-i}.$ Moreover

\begin{proposition}[Theorem $2.5$ in \cite{bus09}]
Each of the $\Delta^\bv_{i,d-1-i}(\T,\X)$ is independent of the choice of the decomposition \eqref{dolo} modulo $\langle P_{2,1}(\T,\X),\,Q_{d-2,1}(\T,\X)\rangle$, and belongs to $\kK.$
\end{proposition}

\smallskip
In connection with the matrices $\bM_i$ defined above, we recall here the  matrix construction for the resultant given in \cite[$3.11.19.7$]{jou97}. For a fixed, $i,\,1\leq i\leq d-4,$ we set
${\mathbf M}_i$ the $(d-2)\times (d-2)$ square matrix, defined as follows:
\begin{equation}\label{morly}
{\mathbf M}_i=\left(\begin{array}{cc}
\bM_i(1)&{\mathbf{Mor}(i)}\\
{\bf 0}& \bM_{d-2-i}(1)^t
\end{array}
\right),
\end{equation}
where $\bM_j(1)$ is the submatrix of $\bM_j$ where we have eliminated the last column, and the matrix $\mathbf{Mor}(i)$ has its rows (resp. columns )indexed by all $\T$ monomials of total degree $d-2-i$ (resp. $i$), in such a way that the entry $\mathbf{Mor}(i)_{\bv,\bv'}$is equal to the coefficient of $\T^{\bv'}\S^{\bv}$ in $\mbox{ Mor}(\S,\T,\X)$ defined in \eqref{morr}. With this notation, we easily deduce that
\begin{equation}\label{kohinor}
F^\bv_{d-2-|\bv|,2}(\T,\X)=\sum_{|\bv'|=d-2-|\bv|}\mathbf{Mor}(i)_{\bv',\bv}\T^{\bv'}
\end{equation}
\begin{proposition}[Proposition $3.11.19.21$ in \cite{jou97}]\label{grr}
$$\big|{\mathbf M}_i\big|=\cE_d(\X).$$
\end{proposition}
\smallskip
To prove our main result, we will need the following technical lemma.
\begin{lemma}\label{sub}
Let $K$ be a field, $n,\, N\in\N$ and $\omega_0, \omega_1,\ldots,\omega_{n-2},\,\tau_1,\ldots,\tau_N\in K^n,$ such that
$\dim(\omega_0,\,\omega_1,\,\ldots,\omega_{n-2})=n-1$, and for each $j=1,\ldots, N,$
$$\dim_K\left(\omega_0,\,\omega_1,\,\ldots,\omega_{n-2},\,\tau_j\right)\leq n-1$$
(where $({\bf F})$ denotes the $K$-vector space generated by the sequence ${\bf F}$).
Then, for each $i,j,\,1\leq i,\,j\leq N,$ we have
$$\dim_K\left(\omega_1,\,\ldots,\omega_{n-2},\,\tau_i,\,\tau_j\right)\leq n-1.
$$
\end{lemma}
\begin{proof}
Suppose that the claim is false. Then,  we will have $\left(\omega_1,\,\ldots,\omega_{n-2},\,\tau_i,\,\tau_j\right)=K^ n$ for some $i,\,j$ and by applying Grassman's formula for computing the dimension of a sum of vector subspaces:
$$\begin{array}{ccl}
\dim_K\left(\omega_1,\,\ldots,\omega_{n-2},\,\tau_i,\,\tau_j\right)&\leq&
\dim_K\left(\omega_0,\,\omega_1,\,\ldots,\omega_{n-2},\,\tau_i\right)+\dim_K\left(\omega_0,\omega_1,\,\ldots,\omega_{n-2},\,\tau_j\right)\\
&&-\dim_K\left(\omega_0,\omega_1,\,\ldots,\omega_{n-2}\right)\leq 2(n-1)-(n-1)=n-1,
\end{array}
$$
a contradiction.
\end{proof}
\bigskip
\subsection{Minimal generators}
Now we are ready to present the main result of this section.
\begin{theorem}\label{mt2}
If $\mu=2$ and the curve $\C$ has all its singularities having multiplicity $2,$ then the following family of $\frac{(d+1)(d-4)}2+5$ polynomials
$$\{\cE_d,\, P_{2,1}(\T,\X),\,Q_{d-2,1}(\T,\X),\,\Delta^{(1,0)}(\T,\X),\,\Delta^{(0,1)}(\T,\X)\} \bigcup \{\Delta^\bv_{i,d-1-i}(\T,\X)\}_{1\leq i\leq d-4, |\bv|=d-2-i}
$$
is a minimal set of generators of $\kK$.
\end{theorem}

\begin{proof}
In \cite[Proposition $3.2$]{bus09}, it is shown that ${\mathrm F}\cup \{\Delta^{0,0}(\T,\X)\}$ is a set of generators of $\kK$, and we just saw in \eqref{descarte} that we can remove
$\Delta^{0,0}(\T,\X)$ from the list. So we only need to prove that this family is minimal, i.e. that there are no superfluous combinations. Apart from $\cE_d(\X),\,P_{2,1}(\T,\X),\,Q_{d-2,1}(\T,\X),$ note that the rest of elements in ${\mathrm F}$ have total degree in $(\T,\X)$ equal to $d-1$. The only generator whose total degree is lower than or equal to $d-1$ is $P_{2,1}(\T,\X)$. So, due to bihomogeneity of the generators, the proof will be done if we just show that
\begin{itemize}
\item $\Delta^{(1,0)}(\T,\X)$  and $\Delta^{(0,1)}(\T,\X)$ are $\K$-linearly independent modulo  $P_{2,1}(\T,\X)$;
\item for each $i=1,\ldots, d-4,\,$ the set $\{\Delta^\bv_{i,d-1-i}(\T,\X)\}_{|\bv|=d-2-i}$ is $\K$-linearly independent modulo  $P_{2,1}(\T,\X)$.
\end{itemize}
To prove the first claim, suppose we have $\lambda_0,\,\lambda_1\in\K$ such that
$$\lambda_0\Delta^{(1,0)}(\T,\X)+\lambda_1\Delta^{(0,1)}(\T,\X)=0 \ \mod P_{2,1}(\T,\X).
$$
Recall also from \eqref{descarte}, that we have
$$T_0\Delta^{(1,0)}(\T,\X)-T_1\Delta^{(0,1)}(\T,\X)=0 \ \mod P_{2,1}(\T,\X).$$
From these two identities, we get
$$(\lambda_1T_0-\lambda_0T_1) \Delta^{(0,1)}(\T,\X)\in\langle P_{2,1}(\T,\X)\rangle,
$$ i.e. $ \Delta^{(0,1)}(\T,\X)\in\langle P_{2,1}(\T,\X)\rangle.$ But this is impossible as \eqref{descarte} shows that $T_1 \Delta^{(0,1)}(\T,\X)=\Delta^{(0,0)}(\T,\X),$ and the latter
being an element different from zero (the ``discrete jacobian'' )in the quotient ring $\K[\T,\X]$ modulo $P_{2,1}(\T,\X),\,Q_{d-2,1}(\T,\X),$
see for instance \cite[$2.1$]{bus09} So, $\lambda_0=\lambda_1=0$ and the claim follows.
\par\smallskip
Choose now $i$ such that $1\leq i\leq d-4,$ and consider the family  $\{\Delta^\bv_{i,d-1-i}(\T,\X)\}_{|\bv|=d-2-i}.$ Suppose that there is a non trivial linear combination $$\sum_{|\bv|=d-2-i}\lambda_\bv\Delta^\bv_{i,d-1-i}(\T,\X)=0 \ \mod\, P_{2,1}(\T,\X),$$
with $\lambda_\bv\in\K\,\forall\bv.$ By the definition of the polynomials  $\Delta^\bv_{i,d-1-i}(\T,\X),$ this last identity implies that the following square extended matrix
$$
(\bM_i|\,\underline{\lambda})=
\left(\begin{array}{cccccc}
L^0_1(\X)&0&0&\ldots& F^{(d-2-i,0)}_{i,2}(\T,\X)&\lambda_{(d-2-i,0)}\\
 L^*_1(\X)& L^0_1(\X)& 0&\ldots& F^{(d-3-i,1)}_{i,2}(\T,\X)&\lambda_{(d-3-i,1)}\\
  L^1_1(\X)&L^*_1(\X)& L^0_1(\X)&\ldots& F^{(d-4-i,2)}_{i,2}(\T,\X)&\lambda_{(d-4-i,2)}\\
\vdots&\ddots & \ddots& \dots &\vdots&\vdots \\ \\
0&\ldots && L^1_1(\X)& F^{(0,d-2-i)}_{i,2}(\T,\X)&\lambda_{(0,d-2-i)}
\end{array}
\right)
$$
is rank-deficient modulo $P_{2,1}(\T,\X).$ We claim the matrix which results by eliminating the second to the last column has maximal rank. Indeed, if this were not the case, by looking at the Sylvester-type structure of the matrix, and performing linear combinations of the columns of this rectangular matrix, we would deduce identity of the form
$$\sum_{|\bv|=d-2-i}\lambda_\bv\T^\bv=\frac{A(\T,\X)}{B(\X)}P_{2,1}(\T,\X)
$$
with $A(\T,\X)\in\K[\T,\X],\,B(\X)\in\K[\X]$. But this is impossible, since from
$$B(\X)\left(\sum_{|\bv|=d-2-i}\lambda_\bv\T^\bv\right)=A(\T,\X)P_{2,1}(\T,\X)$$ we would deduce that $P_{2,1}(\T,\X)$ is not irreducible, which is a contradiction. Hence, these columnas are $\K[\X]$-linearly independent.  By expanding the determinant of the rank-deficient matrix $(\bM_i|\,\underline{\lambda})$ by the second to the last column, and using \eqref{kohinor}, we get
$$0=\sum_{|\bv'|=i}\left|\begin{array}{cccccc}
L^0_1(\X)&0&0&\ldots& \mathbf{Mor}(i)_{\bv',(d-2-i,0)}&\lambda_{(d-2-i,0)}\\
 L^*_1(\X)& L^0_1(\X)& 0&\ldots& \mathbf{Mor}(i)_{\bv',(d-3-i,1)}&\lambda_{(d-3-i,1)}\\
  L^1_1(\X)&L^*_1(\X)& L^0_1(\X)&\ldots& \mathbf{Mor}(i)_{\bv',(d-4-i,2)}&\lambda_{(d-4-i,2)}\\
\vdots&\ddots & \ddots& \dots &\vdots&\vdots \\ \\
0&\ldots && L^1_1(\X)& \mathbf{Mor}(i)_{\bv',(0,d-2-i)}&\lambda_{(0,d-2-i)}
\end{array}\right|\T^{\bv'},
$$
so we conclude that
$$\left|\begin{array}{cccccc}
L^0_1(\X)&0&0&\ldots& \mathbf{Mor}(i)_{\bv',(d-2-i,0)}&\lambda_{(d-2-i,0)}\\
 L^*_1(\X)& L^0_1(\X)& 0&\ldots& \mathbf{Mor}(i)_{\bv',(d-3-i,1)}&\lambda_{(d-3-i,1)}\\
  L^1_1(\X)&L^*_1(\X)& L^0_1(\X)&\ldots& \mathbf{Mor}(i)_{\bv',(d-4-i,2)}&\lambda_{(d-4-i,2)}\\
\vdots&\ddots & \ddots& \dots &\vdots&\vdots \\ \\
0&\ldots && L^1_1(\X)& \mathbf{Mor}(i)_{\bv',(0,d-2-i)}&\lambda_{(0,d-2-i)}
\end{array}\right|=0
$$
for all $\bv',\,|\bv'|=i.$
Lemma \ref{sub} above then implies that
\begin{equation}\label{kori}
\left|\begin{array}{cccccc}
L^0_1(\X)&0&0&\ldots& \mathbf{Mor}(i)_{\bv',(d-2-i,0)}& \mathbf{Mor}(i)_{\bv'',(d-2-i,0)}\\
 L^*_1(\X)& L^0_1(\X)& 0&\ldots& \mathbf{Mor}(i)_{\bv',(d-3-i,1)}&\mathbf{Mor}(i)_{\bv'',(d-3-i,1)}\\
  L^1_1(\X)&L^*_1(\X)& L^0_1(\X)&\ldots& \mathbf{Mor}(i)_{\bv',(d-4-i,2)}&\mathbf{Mor}(i)_{\bv'',(d-4-i,2)}\\
\vdots&\ddots & \ddots& \dots &\vdots&\vdots \\ \\
0&\ldots && L^1_1(\X)& \mathbf{Mor}(i)_{\bv',(0,d-2-i)}&\mathbf{Mor}(i)_{\bv'',(0,d-2-i)}
\end{array}\right|=0
\end{equation}
for any pair $\bv',\,\bv''$ such that $|\bv'|=|\bv''|=i.$  If we compute the determinant of the matrix ${\mathbf M}_i$ defined in \eqref{morly} by Laplace expansion along the first block of rows $\big(\bM_i(1)\ \ {\mathbf{Mor}(i)}\big),$ then due to the zero-block structure of this matrix, it is easy to see that the only non zero minors contributing to this Laplace expansion coming from this block are of the form \eqref{kori}. This implies then that $|{\mathbf M}_i|=0,$ which contradicts Proposition \ref{grr}. Hence, there cannot be
a non trivial linear combination of the form $\sum_{|\bv|=d-2-i}\lambda_\bv\Delta^\bv_{i,d-1-i}(\T,\X)=0 \ \mod\, P_{2,1}(\T,\X),$ and this completes the proof.
\end{proof}

\bigskip
\section{What about $\mu\geq3$?}\label{final}
One may wonder to what extent what we have done in this text for curves with $\mu=2$ can be extended with the same techniques for larger values of $\mu$. We have worked out several examples with {\tt Macaulay 2}, and the situation does not seem to be straightforwardly generalizable. For instance, there will be no statement equivalent to what we obtained in Theorems \ref{mt1o} and \ref{mt1e} for $\mu\geq3,$ where once you fixed the degree $d$ of the curve with a very singular point, the bidegrees of the minimal generators of $\kK$ are determined by it for $\mu=2.$
\par
Indeed, consider the two following $\mu$-bases:
$$\begin{array}{l}
F_{3,1}(\T,\X)=T_0^3X_0+(T_1^3-T_0T_1^2)X_1\\
F_{7,1}(\T,\X)=(T_0^6T_1-T_0^2T_1^5)X_0+(T_0^4T_1^3+T_0^2T_1^5)X_1+(T_0^7+T_1^7)X_2
\\ \\
\tilde{F}_{3,1}(\T,\X)=(T_0^3-T_0^2T_1)X_0+(T_1^3+T_0T_1^2-T_0T_1^2)X_1 \\ 
\tilde{F}_{7,1}(\T,\X)=(T_0^6T_1-T_0^2T_1^5)X_0+(T_0^4T_1^3+T_0^2T_1^5)X_1+(T_0^7+T_1^7)X_2.
\end{array}$$
Each of them parametrizes properly a rational plane curve of degree $10$ having $(0:0:1)$ as a very singular point. However, an explicit computation of a family of minimal generators of $\kK$ for the first curve gives in both cases families of cardinality $10$, but in the first one the generators appear in bidegrees 
$$(3,1),\,(7,1),\,(2,3),\,(2,3),\,(4,2),\,(2,4),\,(1,6),\,(1,6),\,(1,6),\,(0,10),$$
while in the second curve, the generators have bidegrees
$$(3,1),\,(7,1),\,(2,3),\,(2,3),\,(4,2),\,(2,4),\,{\bf (1,5)},\,(1,6),\,(1,6),\,(0,10).$$
Also, the family we can get from \eqref{family} only detects the elements in bidegree $$(3,1),\,(7,1),\,(4,2),\,(0,10),$$ so it will not be true anymore that for $\T$-degrees larger than $\mu-1$, this set actually gives all the generators of $\kK$.  
\par
All this shows that, for $\mu\geq3,$ more information from the curve apart from $(d,\,\mu)$ and if it has a very singular point or not, must be taken into account to get a precise description of the minimal generators of $\kK$. Note also that in the case of mild singularities, the set of elements of $\kK$ proposed by Bus\'e in \cite{bus09} do not generate the whole ideal, and by computing  concrete examples, we find that they almost never neither contain nor are contained in  a minimal set of generators of $\kK.$


\bigskip
\begin{thebibliography}{XXXXXX}
%
%
\bibitem[BH93]{BH93}
 Bruns, Winfried; Herzog, J\"urgen.
 \newblock{\em  Cohen-Macaulay rings.\/} 
 \newblock Cambridge Studies in Advanced Mathematics, 39. Cambridge University Press, Cambridge, 1993.

\bibitem[Bus09]{bus09}
Bus\'e, Laurent.
\newblock{\em On the equations of the moving curve ideal of a rational algebraic plane curve.\/}
\newblock  J. Algebra  321  (2009),  no. 8, 2317--2344.
%
%

%
%
\bibitem[BJ03]{BJ03}
Bus\'e, Laurent; Jouanolou, Jean-Pierre.
\newblock{\em  On the closed image of a rational map and the implicitization problem.\/}
\newblock  J. Algebra  265  (2003),  no. 1, 312--357.

\bibitem[CA00]{cas00} 
Casas-Alvero, Eduardo.
\newblock{\em Singularities of plane curves.\/} 
\newblock London Mathematical Society Lecture Note Series, 276. Cambridge University Press, Cambridge, 2000. 

\bibitem[CCL05]{CCL05}
Chen, Falai; Cox, David; Liu, Yang.
\newblock{\em The $\mu$-basis and implicitization of a rational parametric surface.\/}
\newblock  J. Symbolic Comput.  39  (2005),  no. 6, 689--706.

\bibitem[CWL08]{CWL08}
Chen, Falai; Wang, Wenping; Liu, Yang.
\newblock {\em Computing singular points of plane rational curves.\/}
\newblock J. Symbolic Comput.  43  (2008),  no. 2, 92--117.

\bibitem[CD10]{CD10}
Cortadellas Ben\'itez, Teresa; D'Andrea, Carlos.
\newblock{\em Minimal generators of the defining ideal of the Rees Algebra associated to monoid parametrizations.\/}
\newblock Computer Aided Geometric Design, Volume 27, Issue 6, August 2010, 461--473.


\bibitem[CD13]{CD13}
Cortadellas Ben\'itez, Teresa; D'Andrea, Carlos.
\newblock{\em Rational plane curves parametrizable by conics.\/}
\newblock J. Algebra 373 (2013) 453--480.

\bibitem[CD13b]{CD13b}
Cortadellas Ben\'itez, Teresa; D'Andrea, Carlos.
\newblock{\em Minimal generators of the defining ideal of the Rees Algebra associated to a rational plane parameterization with $\mu=2$ .\/}
\newblock {\tt arXiv:1301.6286 }


\bibitem[Cox08]{cox08}
Cox, David A.
\newblock {\em The moving curve ideal and the Rees algebra.\/}
\newblock Theoret. Comput. Sci.  392  (2008),  no. 1-3, 23--36.

\bibitem[CGZ00]{CGZ00}
Cox, David; Goldman, Ronald; Zhang, Ming.
\newblock{\em On the validity of implicitization by moving quadrics of rational surfaces with no base points.\/}
\newblock  J. Symbolic Comput.  29  (2000),  no. 3, 419--440.

\bibitem[CHW08]{CHW08}
Cox, David; Hoffman, J. William; Wang, Haohao.
\newblock {\em Syzygies and the Rees algebra.\/}
\newblock J. Pure Appl. Algebra  212  (2008),  no. 7, 1787--1796.

\bibitem[CKPU11]{CKPU11}
Cox,David; Kustin, Andrew; Polini, Claudia; Ulrich, Bernd. 
\newblock{\em A study of singularities on rational curves via syzygies.\/}
\newblock To appear in Memoirs of AMS.



\bibitem[CLO07]{CLO07}
Cox, David; Little, John; O'Shea, Donal.
\newblock{\em Ideals, varieties, and algorithms. An introduction to computational algebraic geometry and commutative algebra.\/}
\newblock Third edition. Undergraduate Texts in Mathematics. Springer, New York, 2007.

\bibitem[CSC98]{CSC98}
Cox, David A.; Sederberg, Thomas W.; Chen, Falai.
\newblock {\em The moving line ideal basis of planar rational curves.\/}
\newblock Comput. Aided Geom. Design  15  (1998),  no. 8, 803--827.

%
%
%
%
%

\bibitem[HS12]{HS12}
Hassanzadeh, Seyed Hamid; Simis, Aron.
\newblock{\em Implicitization of the Jonqui\`eres parametrizations.\/}
\newblock {\tt arXiv:1205.1083}

\bibitem[HSV08]{HSV08}
Hong, Jooyoun; Simis, Aron; Vasconcelos, Wolmer V.
\newblock{\em On the homology of two-dimensional elimination.\/}
\newblock  J. Symbolic Comput.  43  (2008),  no. 4, 275--292.

\bibitem[HSV09]{HSV09}
Hong, Jooyoun; Simis, Aron; Vasconcelos, Wolmer V.
\newblock{\em Equations of almost complete intersections.\/}
\newblock  Bulletin of the Brazilian Mathematical Society, 
June 2012, Volume 43, Issue 2, 171--199.

\bibitem[HW10]{HW10}
Hoffman, J. William; Wang, Haohao.
\newblock{\em Defining equations of the Rees algebra of certain parametric surfaces.\/}
\newblock Journal of Algebra and its Applications, Volume: 9, Issue: 6(2010), 1033--1049 



%
%
\bibitem[Jou97]{jou97}
 Jouanolou, J. P. 
 \newblock{\em Formes d'inertie et r\'esultant: un formulaire.\/}
 \newblock Adv. Math. 126 (1997), no. 2, 119--250.

\bibitem[KPU09]{KPU09}
Kustin, Andrew R.; Polini, Claudia; Ulrich, Bernd.
\newblock{\em Rational normal scrolls and the defining equations of Rees Algebras.\/}
\newblock  J. Reine Angew. Math. 650 (2011), 23--65.

\bibitem[KPU13]{KPU13}
Kustin, Andrew; Polini, Claudia; Ulrich, Bernd. 
\newblock{\em The bi-graded structure of Symmetric Algebras with applications to Rees rings.\/}
\newblock {\tt arXiv:1301.7106 }.

\bibitem[Mac]{mac}
Grayson, Daniel R.; Stillman, Michael E.
\newblock{\em Macaulay 2, a software system for research in algebraic geometry.\/}
\newblock Avabilable at {\tt http://www.math.uiuc.edu/Macaulay2/}


%
\bibitem[SC95]{SC95}
Sederberg, Thomas; Chen, Falai.
\newblock{\em Implicitization using moving curves and surfaces.\/}
\newblock Proceedings of SIGGRAPH, 1995, 301--308.

\bibitem[SGD97]{SGD97}
Sederberg, Tom; Goldman, Ron; Du, Hang.
\newblock{\em  Implicitizing rational curves by the method of moving algebraic curves.\/}
\newblock Parametric algebraic curves and applications (Albuquerque, NM, 1995).
\newblock J. Symbolic Comput.  23  (1997),  no. 2-3, 153--175.

%
%
%
%
\bibitem[SWP08]{SWP08}
Sendra, J. Rafael; Winkler, Franz; P\'erez-D\'iaz, Sonia.
\newblock{\em Rational algebraic curves. A computer algebra approach.\/}
\newblock Algorithms and Computation in Mathematics, 22. Springer, Berlin, 2008.

%
%
%
%
\bibitem[Wol10]{math}
Wolfram Research, Inc.
\newblock{\em Mathematica,\/}
\newblock Version 8.0, Champaign, IL (2010).

\bibitem[ZCG99]{ZCG99}
Zhang, Ming; Chionh, Eng-Wee; Goldman, Ronald N.
\newblock{\em On a relationship between the moving line and moving conic coefficient matrices.\/}
Comput. Aided Geom. Design  16  (1999),  no. 6, 517--527.
\end{thebibliography}
\end{document}